\documentclass[oneside,english]{amsart}
\usepackage[T1]{fontenc}
\usepackage[latin9]{inputenc}
\usepackage{babel}
\usepackage{mathrsfs}
\usepackage{mathtools}
\usepackage{amsbsy}
\usepackage{amstext}
\usepackage{amsthm}
\usepackage{amssymb}
\usepackage{stackrel}
\usepackage{graphicx}
\usepackage{epstopdf}
\usepackage[unicode=true,pdfusetitle,
 bookmarks=true,bookmarksnumbered=true,bookmarksopen=true,bookmarksopenlevel=1,
 breaklinks=false,pdfborder={0 0 1},backref=false,colorlinks=false]
 {hyperref}

\usepackage{booktabs}

\makeatletter
\numberwithin{equation}{section}
\numberwithin{figure}{section}
\theoremstyle{definition}
\newtheorem{defn}{\protect\definitionname}
\theoremstyle{remark}
\newtheorem{rem}{\protect\remarkname}
\theoremstyle{plain}
\newtheorem{thm}{\protect\theoremname}
\ifx\proof\undefined\
  \newenvironment{proof}[1][\proofname]{\par
    \normalfont\topsep6\p@\@plus6\p@\relax
    \trivlist
    \itemindent\parindent
    \item[\hskip\labelsep
          \scshape
      #1]\ignorespaces
  }{%
    \endtrivlist\@endpefalse
  }
  \providecommand{\proofname}{Proof}
\fi
\theoremstyle{plain}
\newtheorem{lem}{\protect\lemmaname}
\theoremstyle{plain}
\newtheorem{cor}{\protect\corollaryname}
\theoremstyle{definition}
\newtheorem{example}{\protect\examplename}

\usepackage{caption} 

\usepackage{babel}

\makeatother

\providecommand{\corollaryname}{Corollary}
\providecommand{\definitionname}{Definition}
\providecommand{\examplename}{Example}
\providecommand{\lemmaname}{Lemma}
\providecommand{\remarkname}{Remark}
\providecommand{\theoremname}{Theorem}

\begin{document}
\title[The convergence and uniqueness of a nonlinear Markov chain]{The convergence and uniqueness of a discrete-time nonlinear Markov
chain}
\author{Ruowei Li and Florentin M\"unch}
\address{Ruowei Li: Department of Statistics and Data Science, National University of Singapore, 117546 Singapore; Shanghai Center for Mathematical Sciences, Jiangwan
Campus, Fudan University, No. 2005 Songhu Road, 200438 Shanghai, China;
Max Planck Institute for Mathematics in the Sciences, 04103 Leipzig, Germany.}
\email{ruoweili@nus.edu.sg; rwli19@fudan.edu.cn}
\address{Florentin M\"unch: Institute for Mathematics, Leipzig University, 04109 Leipzig, Germany.}
\email{cfmuench@gmail.com}
\begin{abstract}
In this paper, we prove the convergence and uniqueness of a general
discrete-time nonlinear Markov chain with specific conditions. The
results have important applications in discrete differential geometry.
First, we prove the discrete-time
Ollivier Ricci curvature flow $d_{n+1}\coloneqq(1-\alpha\kappa_{d_{n}})d_{n}$ 
converges to a constant curvature metric on a finite weighted graph. As shown in \cite[Theorem 5.1]{M23}, a Laplacian separation principle holds on a locally finite graph with nonnegative Ollivier curvature. We further prove that the Laplacian
separation flow converges to the constant Laplacian solution and generalize
the result to nonlinear $p$-Laplace operators. Moreover, our results
can also be applied to study the long-time behavior in the nonlinear
Dirichlet forms theory and nonlinear Perron-Frobenius theory. Finally, we define the Ollivier Ricci curvature of the nonlinear Markov chain
which is consistent with the classical Ollivier Ricci curvature, sectional
curvature \cite{CMS24}, coarse Ricci curvature on hypergraphs \cite{IKTU21}
and the modified Ollivier Ricci curvature for $p$-Laplace. We also establish
 the convergence results for the nonlinear Markov chain with
nonnegative Ollivier Ricci curvature. 
\end{abstract}

\maketitle
\tableofcontents{}

\section{Introduction}

A nonlinear Markov chain, introduced by McKean \cite{M66} to tackle
mechanical transport problems, is a discrete space dynamical system
generated by a measure-valued operator that preserves positivity. Compared with the linear Markov chain, its
transition probability is dependent not only on the state but also
on the distribution of the process.

Understanding the long-time behavior of Markov chains is a fundamental problem. A classical result is that an irreducible lazy linear Markov
chain converges to its unique stationary distribution in the total
variation distance \cite{LP17,S09}. For the nonlinear case, Kolokoltsov
\cite{K10} and BA Neumann \cite{N23} 
studied the long-term behavior
of nonlinear Markov chains defined on probability simplex whose transition
probabilities are a family of stochastic matrices. Long-term results
exist for specific continuous-time Markov chains associated with pressure
and resistance games \cite{KM19} and ergodicity criteria for discrete-time
Markov processes \cite{B14,S16}. This paper establishes convergence
and uniqueness results for a general discrete-time nonlinear Markov
chain $P:\Omega\to\Omega$ under some of the following specific conditions:
\begin{itemize}
	\item Conditions on the domain
	\begin{itemize}
		\item[(A)] $\Omega\subseteq\mathbb{R}^{N}$ is closed.
		\item[(B)] 
        $\Omega+r\cdot \overrightarrow{1}=\Omega$ for all $r\in\mathbb{R}$, where  $\overrightarrow{1}=(1,\ldots,1)\in \mathbb{R}^{N}$.
	\end{itemize}
    \end{itemize}
   
    We now introduce the following properties for all $f,g\in \Omega$:
    \begin{itemize}
    \item Basic properties
	\begin{itemize}
			\item Monotonicity
			\begin{itemize}
				\item[(1)] Monotonicity: $Pf\geq Pg$ if $f\geq g$,
                where $f\geq g$ means $f(x)\geq g(x)$ for all components $x = 1,2,\ldots,N.$
				
				\item[(2)] Strict monotonicity of corresponding components: $Pf(x)>Pg(x)$ if $f\geq g$ and $f(x)>g(x)$
				for some component $x \in \{1, \ldots, N\}$.
				
				\item[(3)] Uniform strict monotonicity: $Pf\geq Pg+\epsilon_{0}(f-g)$ if
				$f\geq g$ for some fixed positive $\epsilon_{0}$.
			\end{itemize}
			\item Additivity
			\begin{itemize}
				\item[(4)] Constant additivity: $P(f+C\cdot \overrightarrow{1})=Pf+C\cdot \overrightarrow{1}$, where $C \in \mathbb{R}$ is a constant.
	    	\end{itemize}
    		\item Non-expansion
        	\begin{itemize}
    		\item[(5)] Non-expansion: $\left\Vert Pf-Pg\right\Vert _{\ell^{\infty}}\leq\left\Vert f-g\right\Vert _{\ell^{\infty}}$
    		for all $f,g\in\Omega$.
    	   \end{itemize}
       \end{itemize}
		\item Connectedness
		  \begin{itemize}
		  	\item[(6)] Connectedness: there exists $n_{0}\in\mathbb{N}_{+}$ such that
		  	for every component $x,$ and $f\geq g$ with $f(x)>g(x)$,
		  	we have $P^{n_{0}}f>P^{n_{0}}g$, (i.e., the strict inequality holds component-wise).
		  	\item[(7)] Uniform connectedness: there exists $n_{0}\in\mathbb{N}_{+}$,
		  	positive $\epsilon_{0}$ such that for every component $x$, positive
		  	$\delta$ and $f\geq g+\delta\cdot1_{x}$ (where $1_{x}\in \mathbb{R^{N}}$ and $1_{x}(x)=1$, and $1_{x}(y)=0$ for $y\neq x$),
		  	we have $P^{n_{0}}f\geq P^{n_{0}}g+\epsilon_{0}\delta$.
		  \end{itemize}
	    \item  Accumulation points
	      \begin{itemize}
	      	\item[(8)] Accumulation point at infinity: there exists a component $x_{0} \in \{1, \ldots, N\}$ such that $f_{n}\coloneqq P^{n}f-P^{n}f(x_{0})\cdot \overrightarrow{1}$
	      	has a finite accumulation point $g$, i.e. for every $n\in\mathbb{N}_{+}$
	      	and positive $\epsilon$, there exists $N>n$ such that $\left\Vert f_{N}-g\right\Vert _{\ell^{\infty}}<\epsilon$.
	      	\item[(9)] Finite accumulation point: $P^{n}f$ has a finite accumulation
	      	point $g$. 
	      \end{itemize}
	\end{itemize}

\begin{defn}
A discrete-time nonlinear Markov chain $P:\Omega\to\Omega$ is a map
satisfying monotonicity (1) and non-expansion (5) where $\Omega$ satisfies
(A) and (B). 
\end{defn}
In the theorems, we always reiterate the assumptions (1) and (5),
even though the conditions are implicitly given by the definition. 
\begin{rem}
\label{rem: NLMC }(a) For a linear Markov chain, monotonicity (1) and strict monotonicity of corresponding components
(2) imply that $P$ is lazy, meaning it remains in the same
state with positive probabilities.

(b) Uniform strict monotonicity (3) is stronger than monotonicity
(1) and strict monotonicity of corresponding components (2), which means that (3) implies (1)
and (2).

(c) Monotonicity (1) and constant additivity (4) imply the property of positivity preservation in McKean's work \cite{M66}.

(d) Since $f\leq g+\left\Vert f-g\right\Vert _{{\infty}}\cdot \overrightarrow{1}$ for all
$f$, $g\in \mathbb{R}^N$, monotonicity (1) and constant additivity (4) imply
the non-expansion condition (5), which is more natural for nonlinear
operators.

(e) A linear Markov chain is called irreducible if for all states $x,y$
there exists some $n$ such that its kernel $P^{n}(x,y)>0$, i.e.
every state can be reached from every other state. Saying a discrete
Markov chain defined on a graph is irreducible is the same as saying
the graph is connected, which is crucial to the uniqueness of the
stationary distribution. Condition (6) is a nonlinear version of the connectedness
condition. Moreover, $P^{n_{0}}$ also satisfies the strict
monotonicity of corresponding components (2).

(f) The assumption of a finite accumulation point for $f_{n}$ (8)
is weaker than (9), as it allows for cases where all components of $P^{n}f$ go to
infinity. Moreover, assumption (8) is necessary.
In subsection \ref{subsec:An-example-of non-convergence}, we provide
a counterexample demonstrating that $P^{n}f(y)-P^{n}f(x)$ may fail to
converge, even within the interval $\left[-\infty,\infty\right]$,
if (8) is not assumed.

(g) Consider a Markov chain $Q$ with maximal
eigenvalue $0<\lambda<1$ and eigenvector $f\in \mathbb{R}^{N}$, i.e., $Qf=\lambda f$. Defining $Pf\coloneqq\text{log}Q\left(\text{exp}f\right)$ (both log and exp are applied componen-twise),
then $P\log f=\log\lambda \cdot \overrightarrow{1}+\log f$,
that is, nonlinear Markov chain $P$ exhibits linear growth with slope $\text{log}\lambda$. 
\end{rem}
We now present our main results. Note that in the following theorems, $\mathbb{R}^{N}$ can be replaced by $\Omega$ satisfying (A) and (B).  We mainly apply Theorem \ref*{thm:NLMC convergence 1} to applications.
\begin{thm}
\label{thm:NLMC convergence 2}Let $f\in \mathbb{R}^{N}$. If a discrete-time nonlinear Markov
chain $P:\mathbb{R}^{N}\to\mathbb{R}^{N}$ satisfies

(1) monotonicity,

(2) strict monotonicity of corresponding components,

(5) non-expansion,

(9) $P^{n}f$ has a finite accumulation point $g\in \mathbb{R}^{N}$,

then $Pg=g$ and $P^{n}f\to g$ as $n\to\infty$. 
\end{thm}
Then we give the second convergence result. 
\begin{thm}
\label{thm:NLMC convergence 1}Let $f\in \mathbb{R}^{N}$. If a discrete-time nonlinear Markov
chain $P:\mathbb{R}^{N}\to\mathbb{R}^{N}$ satisfies

(1) monotonicity,

(2) strict monotonicity of corresponding components,

(4) constant additivity,

(8) accumulation point at infinity, i.e., there exists a component $ x_{0}\in \{1, \ldots, N\}$ such that $f_{n}\coloneqq P^{n}f-P^{n}f(x_{0})\cdot \overrightarrow{1}$
has a finite accumulation point $g\in \mathbb{R}^{N}$, 

then $f_{n}\to g$ as $n\to\infty$.
Moreover, if $P$ also satisfies

(6) connectedness,

then the convergence limit is unique. That is, for any other sequence $\tilde{f_{n}}:=P^{n}\tilde{f}-P^{n}\tilde{f}(\tilde{x})\cdot \overrightarrow{1}$ with $\tilde{f}\in \mathbb{R}^{N}$ and $\tilde{x}\in \{1, \ldots, N\}$ (possibly different from $f$ and $x_0$), if it has a finite accumulation point $\tilde{g}$, 
then $\underset{n\to\infty}{\text{lim}}\tilde{f}_{n}=\tilde{g}=g=\underset{n\to\infty}{\text{lim}}f_{n}$. 
\end{thm}
Next, we give another convergence result. 
\begin{thm}
\label{thm:NLMC convergence 3}Let $f\in \mathbb{R}^{N}$. If a discrete-time nonlinear Markov
chain $P:\mathbb{R}^{N}\to\mathbb{R}^{N}$ satisfies

(1) monotonicity,

(5) non-expansion,

(7) uniform connectedness,

(8) accumulation point at infinity, i.e., $f_{n}\coloneqq P^{n}f-P^{n}f(x_{0})\cdot \overrightarrow{1}$
has a finite accumulation point $g\in \mathbb{R}^{N}$, 

then $f_{n}\to g$ as $n\to\infty$
and the limit is unique. 
\end{thm}
\begin{rem}
Theorem 1 proves the convergence of nonlinear Markov chains under
the assumption of finite accumulation points (9). But Theorem 2 and
Theorem 3 include the case of accumulation points at infinity (8),
that is, all components of $P^{n}f$ go to infinity. While Theorem
2 needs a stronger constant additivity condition (4), Theorem 3 needs
a stronger uniform connectedness condition (7). 
\end{rem}
The convergence results have important applications in discrete differential
geometry which has become a hot research subject in the last decade.
Curvature quantifies how a geometric object deviates from
a flat space in Riemannian Geometry
\cite{J95}, and various discrete analogs on graphs \cite{EM12,JL14,LY10,M13,O07,O09,S99,LLY11,forman2003bochner}
have attracted notable interest. Among them, the idea of discrete Ollivier Ricci curvature $\kappa(x,y)=1-\frac{W(\mu_{x},\mu_{y})}{d(x,y)}$
is based on the comparison between the Wasserstein distance $W$ of probability measures
$\mu_{x},\mu_{y}$ over the one-step neighborhoods of vertices $x,y$ and the distance 
$d(x,y)$ between the centers \cite{O07,O09}. 
Lin, Lu, and Yau modified this notion in \cite{LLY11} to a limiting version that is more suitable for graphs.

Ricci flow on a Riemannian manifold, introduced by Hamilton \cite{Hamilton1982}, is a process that 
smooths the metric but may lead to singularities, which can be removed through "surgery" to continue the flow.
Ricci flow (with surgery) played a pivotal role in Perelman's landmark work of solving the Poincar\'e conjecture.
Ricci flow as a powerful method can also be applied to discrete geometry and has drawn significant interest recently. 
Ollivier \cite{O07} suggested defining the continuous time Ricci flow. 
Ni et al. in \cite{NLLG19} claimed good community detection on networks and network alignment using the discrete Ricci flow. Their experimental results indicate the convergence of discrete Ricci flow, though a theoretic proof of this convergence was still open.
Yau et al. \cite{BLLWY20} proved the existence and uniqueness of a normalized continuous-time Ricci flow and obtained several convergence results on path and star graphs. They also emphasized the question: "If the limit object of the Ricci flow exist? Do they have constant curvature?" In this paper, we prove that the discrete-time Ollivier Ricci curvature flow converges to a constant curvature metric.


A weighted graph $G=(V,E,w,m,d)$ consists of the vertex set $V$, the edge set $E$ and the weight functions $m:V \to \mathbb{R}^+$ and $w: E \to \mathbb{R}^+$. And $d:V^2 \to \mathbb{R}_{\geq 0}$ is a path metric function
on graph $G$. We write $x\sim y$ if $x,y\in V$ are connected by an edge. 

For a finite weighted graph $G=(V,E,w,m,d)$ with $\text{Deg}(x):=\frac{1}{m(x)}\sum_{y \sim x} w(x,y)$  $\leq1$ for all $x\in V$. For an initial metric $d_{0}$, fix some $C$ as the deletion threshold such that $C>\underset{x\sim y\sim z}{\max}\frac{d_{0}(x,y)}{d_{0}(y,z)}$. Then we can execute the discrete Ricci flow with surgery algorithm (Algorithm \ref{tbl:algorithm1}).
Since the graph $G$ is finite and the graph of a single edge cannot be deleted, the algorithm terminates after finitely many steps. On each connected component of the final graph $\tilde{G}$, the distance ratios are bounded in $n$, and hence, $\text{log}d_{n}$ has an accumulation point at infinity.
Considering the Ricci flow as a nonlinear Markov chain on each
connected component of $\tilde{G}$, by Theorem \ref{thm:NLMC convergence 1}
we can prove that (\ref{eq: LLL curvature flow}) converges to a constant curvature metric. 
\begin{thm}
\label{thm:LLL curvature flow converges}Let $d_0$ be an initial metric on a finite weighted graph $G = (V, E, w, m, d_0)$ with $\text{Deg}(x)\leq1$ for
all $x\in V$. Through the discrete Ricci flow with surgery (Algorithm \ref{tbl:algorithm1}), $\frac{d_n(e)}{\text{max } d_n(e')}$  converges to a constant-curvature metric on each connected component of the final graph $\tilde{G}$, where the max is taken over all $ e'$ in the same connected component as $e$ on $\tilde{G}$. 
\end{thm}
\begin{rem}
For a general weighted graph \( G = (V, E, w, m, d) \), this algorithm and its convergence results also hold for the Lin-Lu-Yau-Ollivier Ricci curvature (flow). See Definition \ref{def-curvature} in Section 3 for details.
\end{rem}
For another application, the authors in \cite{HM21,M23} consider a
locally finite graph $G=(V,E,w,m,d)$ with a nonnegative Ollivier
curvature, where $V=X\cup K\cup Y$, $K$ is finite and $E(X,Y)=\emptyset$,
that is, there are no edges between $X$ and $Y$. The space of all functions defined on the vertex set $V$ is denoted by $\mathbb{R}^V$.
They want to find
a function with a constant gradient on $X\cup Y$, minimal on $X$
and maximal on $Y$, and the Laplacian of $f$ should be constant
on $K$. By nonnegative Ollivier curvature, it will follow that the
cut set $K$ separates the Laplacian $\Delta f$, i.e., $\Delta f\mid_{X}\geq\text{const}\geq\Delta f\mid_{Y}$,
which is a Laplacian separation principle \cite[Theorem 5.1]{M23}.
The result is crucial for proving an isoperimetric concentration inequality
for Markov chains with nonnegative Ollivier curvature \cite{M23},
a discrete Cheeger-Gromoll splitting theorem \cite{HM21}, and a discrete
positive mass theorem \cite{HMZ23}. Here we prove a natural parabolic
flow converging to the solution $f$. Now we give the details about
the Laplacian separation flow. First, define an extremal 1-Lipschitz
extension operator $S:\mathbb{R}^{K}\to\mathbb{R}^{V}$, 
\[
Sf(x)\coloneqq\begin{cases}
\begin{array}{c}
f(x):\\
\underset{y\in K}{\text{min}}\left(f(y)+d(x,y)\right):\\
\underset{y\in K}{\text{max}}\left(f(y)-d(x,y)\right):
\end{array} & \begin{array}{c}
x\in K,\\
x\in Y,\\
x\in X.
\end{array}\end{cases}
\]
Let $Lip(1,K)\coloneqq{\left\{ f\in\mathbb{R}^{K}:f(y)-f(x)\leq d(x,y),\text{for all }x,y\in K\right\} }$,
where $d$ is the graph distance on $G$. Then $S(Lip(1,K))$$\subseteq Lip(1,V)$.
In \cite{HM21}, it is proven via elliptic methods that there exists
some $g\in \mathbb{R}^{K}$ with $\Delta Sg \mid_K=\text{const}$. Here we give the parabolic
flow $(id+\epsilon\Delta)S$, and show that it converges to the
constant Laplacian solution, assuming nonnegative Ollivier Ricci
curvature. 
\begin{thm}
\label{thm:existence of harmonic fuction}
Let $G$ be a locally finite graph with nonnegative Ollivier curvature, and let $x_0 \in K$. Define 
$P \coloneqq \left((id + \epsilon \Delta) S\right)\big|_K$,
where $\epsilon > 0$ is sufficiently small so that $diag(id + \epsilon \Delta)$ is positive on $C_0(\bar{K})$. Then for any $f \in Lip(1, K)$, there exists $g \in Lip(1, K)$ such that
\[
P^n f - P^n f(x_0) \cdot \overrightarrow{1} \to g,
\]
and
\[
\left.\Delta Sg\right|_X \geq \left.\Delta Sg\right|_K \equiv \text{const} \geq \left.\Delta Sg\right|_Y.
\]
\end{thm}
Then we want to generalize the result to other nonlinear operators on a
locally finite graph $G=(V,E,w,m,d)$,
such as the $p$-Laplace operator, which can be defined as the subdifferential of the energy functional 
\[
\mathscr{E}_{p}(f)=\frac{1}{2}\underset{x,y\in V}{\sum}\frac{w(x,y)}{m(x)}\lvert \nabla_{xy}f \rvert^{p}, \;\forall f \in \mathbb{R}^V,
\]
where $\nabla_{xy}f=f(y)-f(x)$.
More explicitly, the $p$-Laplace operator $\Delta_{p}:\mathbb{R}^{V}\to\mathbb{R}^{V}$
is given by 
\[
\Delta_{p}f(x)\coloneqq\frac{1}{m(x)}\underset{y}{\sum}w(x,y)\lvert \nabla_{xy}f\rvert^{p-2} \nabla_{xy}f,\text{if }p>1,
\]
and 
\[
\Delta_{1}f(x)\in\ \frac{1}{m(x)}\underset{y}{\sum}w(x,y)\,\textrm{sign}(\nabla_{xy}f) \text{, }
\]
where $\text{sign}\ensuremath{\left(t\right)}=\ensuremath{\begin{cases}
\begin{array}{c}
1,\\{}
[-1,1]\\
-1,
\end{array}, & \begin{array}{c}
t>0.\\
t=0.\\
t<0.
\end{array}\end{cases}}$ Note that $p=2$ is the general discrete Laplace operator $\Delta$.

There are two main difficulties. The first arises from the non-smooth behavior of $\Delta_{p}f$ near $\nabla_{xy} f = 0$. For example, the derivative
of $\Delta_{1}f$ near $\nabla_{xy}f=0$ is large, which causes the
operator $id+\epsilon\Delta_{p}$ to fail to maintain the strict monotonicity of corresponding components
condition (2). Our idea is to consider its resolvent $J_{\epsilon}=\left(id-\epsilon\Delta_{p}\right)^{-1}$
instead of the flow $id+\epsilon\Delta_{p}$. The resolvent operator
$J_{\epsilon}$ is single-valued and monotone, and we can check that
$J_{\epsilon}$ satisfies the strict monotonicity of corresponding components condition in Lemma
\ref{lem:strict monotonicity}.

Another difficulty is the need for a new curvature condition to ensure the Lipschitz decay property, which implies compactness, as well as
the existence of accumulation points. Define a new curvature on a
graph $G=(V,E,w,m,d_{0})$ with combinatorial distance $d_{0}$ as
\begin{equation}
\hat{k}_{p}(x,y)\coloneqq\underset{\pi_{p}}{\text{sup}}\underset{x',y'\in B_{1}(x)\times B_{1}(y)}{\sum}\pi_{p}(x',y')\left(1-\frac{d_{0}(x',y')}{d_{0}(x,y)}\right),\label{eq: modified Ollivier Ricci curvature}
\end{equation}
where $\pi_{p}$ satisfies transport plan conditions and we require
$\pi_{p}(x',y')=0$ if $x'=y'$ (forbid 3-cycles) for $p>2$, and
$\pi_{p}(x',y')=0$ if $x'\neq x$ and $y'\neq y$ and $d_{0}(x',y')=2$
(forbid 5-cycles) for $1\leq p<2$. See detailed definition (\ref{eq:new curvature})
in subsection 4.2.

Then the convergence of the nonlinear Laplace separation flow can be proved.
\begin{thm}
\label{thm:existance of p-harmonic}Let $G$ be a locally finite graph with a nonnegative modified curvature $\hat{k}$, and let $x_0 \in K$. Define $P:=\left((id+\epsilon\Delta_{p})S\right)\mid_{K}$, where $\epsilon > 0$ is sufficiently small so that $\text{diag}(id+\epsilon\Delta_{p})$
is positive on $C_0(\bar{K})$. Then for all $f\in Lip(1,K)$, there exists $\tilde{f}\in Lip(1,K)$
such that $$P^{n}f-P^{n}f(x_{0}) \cdot \overrightarrow{1}\to \tilde{f}.$$ Moreover,
there exist $h,g\in\mathbb{R}^{V}$ such that $g\in\Delta_{p}Sh$
and $g\mid_{X}\geq g\mid_{K}\equiv\text{const}\geq g\mid_{Y}$, where $Sh:=S(h|_K)$.
\end{thm}
Moreover, our nonlinear Markov chain settings overlap with the nonlinear
Dirichlet forms theory and nonlinear Perron-Frobenius theory, and
our theorems can be applied well to them. The theory of Dirichlet
forms is conceived as an abstract version of the variational theory
of harmonic functions. For many application fields, such as Riemannian
geometry \cite{J95}, it is necessary to generalize Dirichlet forms to a nonlinear
version. Since the conditions of our theorems fit well in the nonlinear
Dirichlet form theory, with additional accumulation points at infinity
assumptions we can obtain the convergence by Theorem \ref{thm:NLMC convergence 1},
see Theorem \ref{thm: nonlinear Dirichlet form}. The classical Perron-Frobenius
theory concerns the eigenvalues and eigenvectors of nonnegative coefficient
matrices and irreducible matrices. In order to apply the theory to
a more general setting, there has been extensive research on the nonlinear
Perron-Frobenius theory. After some replacement of maps, our convergence
results can also be applied to the nonlinear Perron-Frobenius theory,
see Theorem \ref{thm: nonlinear Perron=00003D002013Frobenius}.

In Section 5, we introduce a definition of Ollivier Ricci curvature
of nonlinear Markov chains based on the Lipschitz decay property.
Namely, for a nonlinear Markov chain $P$ satisfying the properties of (1) monotonicity,
(2) strict monotonicity of corresponding components and (4) constant additivity, let
$d:V^{2}\to[0,+\infty)$ be a distance function. Then for $r>0$,
define 
\[
Ric_{r}(P,d)\coloneqq1-\underset{Lip(f)\leq r}{\text{sup}}\frac{Lip(Pf)}{r},
\]
That is, if $Lip(f):=\sup_{\substack{x\neq y\in V}}\frac{|f(x)-f(y)|}{d(x,y)}=r$, then $Lip(Pf)\leq(1-Ric_{r})Lip(f)$. Since
the nonnegative Ollivier Ricci curvature guarantees the existence of
accumulation points at infinity (8), then as a corollary of Theorem
\ref{thm:NLMC convergence 1}, we can get the convergence results
for the nonlinear Markov chain with a nonnegative Ollivier Ricci curvature.
And we can also define the Laplacian separation flow of a nonlinear
Markov chain with $Ric_{1}(P,d)\geq0$. We further demonstrate that this definition coincides with the classical Ollivier Ricci curvature (\ref{Ollivier_curvature}),
sectional curvature \cite{CMS24}, coarse Ricci curvature on hypergraphs
\cite{IKTU21} and the modified Ollivier Ricci curvature $\hat{k}_{p}$
for $p$-Laplace (\ref{eq: modified Ollivier Ricci curvature}).

\section{Convergence and uniqueness of nonlinear Markov chains}

\subsection{Proofs of main theorems. }

In this section, we give proof ideas and specific proofs of our main
theorems. First, we summarize the proof ideas for Theorem \ref{thm:NLMC convergence 2}.
Let $Lf=Pf-f$ and $\lambda(f)\coloneqq\parallel Lf\parallel_{\infty}$. 
For $n\in \mathbb{N}_0$, since $\lambda(P^{n}f)=\parallel LP^{n}f\parallel_{\infty}$ is decreasing in $n$ and $g$ is a finite accumulation point, then
$\lambda(P^{k}g)=\lambda(g)$ for all $k \in \mathbb{N}_0$. Since $\eta_{+}(Pg)\subseteq\eta_{+}(g)\coloneqq\left\{ 1\leq x\leq N:Lg(x)=\lambda_{+}(g)\right\} $, where $\lambda_{+}(P^{k}g)\coloneqq \max_xLP^{k}g(x)$, 
then there exists some $x$ such that $P^{k}g(x)=g(x)+k\lambda_{}(g)$. Taking
a subsequence $\left\{ k_{i}\right\} $ such that  $P^{k_{i}}g$ is
also a finite accumulation point for any fixed $k_{i}$, implying $\lambda(g)=0$. Then $g$
is a fixed point and $P^{n}f$ converges. The proof details are as
follows. 
\begin{proof}[Proof of Theorem \ref{thm:NLMC convergence 2}]
For every $f\in \mathbb{R}^N$, define $Lf=Pf-f$ and $\lambda(f)\coloneqq\parallel Lf\parallel_{\infty}$. 
For $n\in \mathbb{N}_0$, since $\lambda(P^{n}f)=\parallel LP^{n}f\parallel_{\infty}$ is decreasing in $n$ and $g$ is a finite accumulation point,
then 
\[
\lambda(g)=\underset{n\to\infty}{\text{lim}}\lambda(P^{n}f).
\]
And for $k\in \mathbb{N}_0$,
\[
\lambda(P^{k}g)=\underset{n\to\infty}{\text{lim}}\lambda(P^{n+k}f).
\]
Hence $\lambda(g)=\lambda(P^{k}g)$ for any $k\in \mathbb{N}_0$. For $n\in \mathbb{N}_0$ and $f\in \mathbb{R}^N$, define 
$$\lambda_{+}(P^{n}f)\coloneqq\underset{1\leq x\leq N}{\text{max}}LP^{n}f(x) \text{ and } \lambda_{-}(P^{n}f)\coloneqq\underset{1\leq x\leq N}{\text{min}}LP^{n}f(x),$$ 
then $$\lambda(P^{n}f)=\max\{\lambda_{+}(P^{n}f),-\lambda_{-}(P^{n}f)\}.$$
Define $\eta(f)\coloneqq\left\{ 1\leq x\leq N:Lf(x)=\lambda(f)\right\} $.
For fixed $k\in \mathbb{N}_0$, we may assume without loss of generality that $$\lambda(P^{k}g)=\lambda_{+}(P^{k}g),$$
as the case $\lambda(P^{k}g)=-\lambda_{-}(P^{k}g)$ is similar. Since for any $l\leq k$,
$$\lambda_{}(P^{k}g)=\lambda_{}(P^{l}g)\geq \lambda_{+}(P^{l}g)\geq\lambda_{+}(P^{k}g)=\lambda_{}(P^{k}g),$$
then $\lambda_{+}(P^{l}g)=\lambda_{}(P^{k}g)=\lambda_{}(g)$.
Suppose $$x'\in\eta_{+}(P^{k}g)\coloneqq\left\{ x:LP^{k}g(x)=\lambda_{+}(P^{k}g)\right\} ,$$
then we claim that $x'\in\eta_{+}(P^{k-1}g)$. If not, we have $$LP^{k-1}g(x')<\lambda_{+}(P^{k-1}g)$$
and
\[
P^{k}g(x')<P^{k-1}g(x')+\lambda_{+}(P^{k-1}g).
\]
By strict monotonicity of corresponding components (2) and non-expansion condition (5), we know
\[
P^{k+1}g(x')<P\left(P^{k-1}g+\lambda_{+}(P^{k-1}g)\cdot \overrightarrow{1}\right)(x')\leq P^{k}g(x')+\lambda_{+}(P^{k-1}g).
\]
That implies 
\[
LP^{k}g(x')<\lambda_{+}(P^{k-1}g)\leq\lambda(P^{k-1}g)=\lambda(P^{k}g)=\lambda_{+}(P^{k}g),
\]
which induces a contradiction. Hence, there exists some $x'\in\eta_{+}(P^{l}g)$
for all $l\leq k$. Thus, 
\begin{equation}
P^{k}g(x')=g(x')+k\lambda_{+}(g)=g(x')+k\lambda(g).\label{eq: linear growth}
\end{equation}
Since $g$ is a finite accumulation point, taking a subsequence
$\left\{ k_{i}\right\} $ such that $P^{k_{i}}g$ is still a finite
accumulation point for every fixed $k_{i}$, then $\lambda(g)=0$ by (\ref{eq: linear growth}) and non-expansion property (5).
Hence $P^{k}g=g$ and $P^{n}f\to g$ as $n\to+\infty$. 
\end{proof}
For Theorem \ref{thm:NLMC convergence 1}, without the assumption
of finite accumulation points, the argument of $\lambda(P^{n}f)$
is insufficient. The proof idea is as follows. For $n\in \mathbb{N}_0$ and $f\in \mathbb{R}^N$, define $\lambda_{+}(P^{n}f)\coloneqq\underset{1\leq x\leq N}{\text{max}}LP^{n}f(x)$ and $\lambda_{-}(P^{n}f)\coloneqq\underset{1\leq x\leq N}{\text{min}}LP^{n}f(x)$, 
and prove it is decreasing in $n$. Since $g$ is the accumulation
point, then $\lambda_{+}(P^{k}g)=\lambda_{+}(g)$ for any $k\in \mathbb{N}_0$. By
the conclusion $$\eta_{+}(Pg)\subseteq\eta_{+}(g)\coloneqq\left\{ x:Lg(x)=\lambda_{+}(g)\right\} ,$$
there exists some $x$ such that $LP^{k}g(x)$ attains the maximum,
i.e. $P^{k}g(x)=g(x)+k\lambda_{+}(g)$. And there also exists some
$y$ attaining its minimum, i.e. $P^{k}g(y)=g(y)+k\lambda_{-}(g)$.
Taking a subsequence $\left\{ k_{i}\right\} \subset\mathbb{N}_0$ such that $P^{k_{i}}g$
is also a finite accumulation point for every fixed $k_{i}$, implying $\lambda_{+}(g)=\lambda_{-}(g)$,
that is, $Lg \equiv \text{const}$. Hence $P^{k}g=g+k\lambda_{+}(g) \cdot \overrightarrow{1}$, and $g$ is
the limit of $f_{n}$. For the uniqueness, we prove that the linear
growth rate $\lambda_{+}(g)$ of different accumulation points are the
same by the non-expansion property. Then by the connectedness (6) we
know that all accumulation points are the same. The proof details
are as follows. 
\begin{proof}[Proof of Theorem \ref{thm:NLMC convergence 1}]
Define $Lf\coloneqq Pf-f$ and $\lambda_{+}(f)\coloneqq\underset{x}{\text{max}}Lf(x)$ for every $f\in \mathbb{R}^N$, 
then $$Pf\leq f+\lambda_{+}(f) \cdot \overrightarrow{1}.$$ By monotonicity (1) and constant
additivity (4), we have 
\[
P^{2}f\leq P(f+\lambda_{+}(f)\cdot \overrightarrow{1})=Pf+\lambda_{+}(f) \cdot \overrightarrow{1}.
\]
Hence $\lambda_{+}(Pf)\leq\lambda_{+}(f)$, that is, $\lambda_{+}(P^{n}f)$
is decreasing in $n$. Since $g$ is a finite accumulation point for
$f_{n}$, 
\[
\lambda_{+}(g)=\underset{n\to\infty}{\text{lim}}\lambda_{+}(P^{n}f).
\]
For any fixed $k\in \mathbb{N}_0$, since $P^{k}g$ is an accumulation point for $P^{k}f_{n}$,
and $LP^{k}f_{n}=LP^{n+k}f$ by constant additivity (4), we have 
\[
\lambda_{+}(P^{k}g)=\underset{n\to\infty}{\text{lim}}\lambda_{+}(P^{n+k}f).
\]
Then $\lambda_{+}(g)=\lambda_{+}(P^{k}g).$

Defining the maximum points set as $$\eta_{+}(f)\coloneqq\left\{ 1\leq x\leq N:Lf(x)=\lambda_{+}(f)\right\}, $$
we claim that $\eta_{+}(Pf)\subseteq\eta_{+}(f)$ if $\lambda_{+}(f)=\lambda_{+}(P^{}f)$.
If $Lf(x)<\lambda_{+}(f)$, i.e. $Pf(x)<f(x)+\lambda_{+}(f)$, since
$Pf\leq f+\lambda_{+}(f)\cdot \overrightarrow{1}$, then by strict monotonicity of corresponding components (2) and constant
additivity (4), 
\[
P^{2}f(x)<P(f+\lambda_{+}(f)\cdot \overrightarrow{1})(x)=Pf(x)+\lambda_{+}(f).
\]
That is, we have $LPf(x)<\lambda_{+}(f)=\lambda_{+}(Pf)$, which induces a contradiction. Hence,
$\eta_{+}(Pf)\subseteq\eta_{+}(f)$.

For any $k\in \mathbb{N}_0$, there exists some $x$ such that $x\in\eta_{+}(P^{l}g)$ for all $l\leq k$, that is, 
\[
P^{k}g(x)=g(x)+k\lambda_{+}(g).
\]
By the same argument, for $\lambda_{-}(f)\coloneqq\underset{x}{\text{min}}Lf(x)$,
there is $y$ such that 
\[
P^{k}g(y)=g(y)+k\lambda_{-}(g).
\]
Since $g$ is a finite accumulation point, then there is a subsequence
$\left\{ n_{i}\right\} $ such that $f_{n_{i}}\to g$. Taking a subsequence
$\left\{ k_{i}\right\} \subset\mathbb{N}_0$ with $\left\{ n_{i}+k_{i}\right\} $ and
$\left\{ n_{i}\right\} $ coincide, then $P^{k_{i}}g(x)-P^{k_{i}}g(y)$
must be finite, which implies $\lambda_{+}(g)=\lambda_{-}(g)$ and
$Lg=Pg-g\equiv\lambda_{+}(g)$. Then we get $P^{n}g=g+n\lambda_{+}(g) \cdot \overrightarrow{1}$
and $P^{n}g-P^{n}g(x_{0}) \cdot \overrightarrow{1}=g-g(x_{0}) \cdot \overrightarrow{1}$. Since $g$ (resp. $P^{k}g$)
is an accumulation point for $f_{n}$ (resp. $P^{k}f_{n}$), for any
$\varepsilon >0$, there exists some $n$ such that 
\[
\Vert P^{n}f-P^{n}f(x_{0}) \cdot \overrightarrow{1}-g\Vert_{\infty}<\varepsilon
\]
and 
\[
\Vert P^{n+k}f-P^{n}f(x_{0}) \cdot \overrightarrow{1}-P^{k}g\Vert_{\infty}<\varepsilon.
\]
Replacing $P^{k}g$ by $g+k\lambda_{+}(g) \cdot \overrightarrow{1}$, we have 
\[
\Vert P^{n+k}f-P^{n}f(x_{0}) \cdot \overrightarrow{1}-g-k\lambda_{+}(g) \cdot \overrightarrow{1}\Vert_{\infty}<\varepsilon.
\]
Since $g(x_{0})=0,$ we know $\mid P^{n+k}f(x_{0})-P^{n}f(x_{0})-k\lambda_{+}(g)\mid<\varepsilon.$
Then 
\begin{align*}
 & \Vert f_{n+k}-g\Vert_{\infty}\\
\leq & \Vert P^{n+k}f-P^{n}f(x_{0}) \cdot \overrightarrow{1}-g-k\lambda_{+}(g) \cdot \overrightarrow{1}\Vert_{\infty}+\mid P^{n+k}f(x_{0})-P^{n}f(x_{0})-k\lambda_{+}(g)\mid\\
< & 2\varepsilon,
\end{align*}
which means $f_{n}$ converges to $g$ as $n\to\infty$.

Then we want to prove the uniqueness with the connectedness assumption
(6). By the above argument, for any finite accumulation point $g$,
we know $Pg-g\equiv\text{const}.$ We claim that the constants of any accumulation
points are the same. If not, then there exist $c_{1}\neq c_{2}$ such
that $P^{n}g^{1}=g^{1}+nc_{1}\cdot \overrightarrow{1}$ and $P^{n}g^{2}=g^{2}+nc_{2}\cdot \overrightarrow{1}$. Hence
$\parallel P^{n}g^{1}-P^{n}g^{2}\parallel_{\infty}\to\infty$, which
contradicts to the non-expansion property (5), i.e. $\parallel Pg^{1}-Pg^{2}\parallel_{\infty}\leq\parallel g^{1}-g^{2}\parallel_{\infty}$.
Then we proved the claim.

Next if $g^{1}$ and $g^{2}$ are two different accumulation points,
by adding a constant, w.l.o.g., assume $g^{1}\geq g^{2}$ with $g^{1}(y)=g^{2}(y)$
and $g^{1}(x)>g^{2}(x)$. By the connectedness condition (6), we get
$$g^{1}(y)+n_0c=P^{n_0}g^{1}(y)>P^{n_0}g^{2}(y)=g^{2}(y)+n_0c,$$
which contradicts to $g^{1}(y)=g^{2}(y)$. Hence all accumulation
points are the same, which shows the uniqueness of limits. 
\end{proof}
For Theorem \ref{thm:NLMC convergence 3}, we remind that there is
a naive but fatal idea of considering $\tilde{P}f\coloneqq Pf-Pf(x_{0})\cdot \overrightarrow{1}$,
which has a finite accumulation point, but may lack the required monotonicity
(1) and non-expansion (5) properties.

Our proof idea is as follows. Lemma \ref{lem: convergence lemma by the uniqueness}
states that a sequence ${\left\{ x_{n}\right\} }$ converges if it
has exactly one accumulation point and satisfies $d(x_{n},x_{n+1})\leq C$
for all $n \in \mathbb{N}_+$. Then we prove Theorem \ref{thm:NLMC convergence 3}
by proving the uniqueness of accumulation points. Let $\tilde{P}=P^{n_{0}}$,
satisfying non-expansion (5) and uniform connectedness (7) with $n_{0}=1$.
Suppose there is a subsequence $\left\{ n_{k}=m_{k}n_{0}+i\right\} $
with $0\leq i\leq n_{0}-1$ such that $f_{n_{k}}\to g$ and $P^{n_{k}}f(x_{0})=\tilde{P}^{m_{k}}P^{i}f(x_{0})\to a\in[-\infty,+\infty]$
as $k\to+\infty$. Divide it into two cases.

In the case of $|a|<+\infty$, since $\tilde{P}^{m}P^{i}f$
has a finite accumulation point $\tilde{g}$, then $\tilde{P}\tilde{g}=\tilde{g}$
by Theorem \ref{thm:NLMC convergence 2}. And accumulation points
are the same after adding a constant by the uniform connectedness
of $\tilde{P}$, which implies the uniqueness of accumulation points
$\tilde{g}-\tilde{g}(x_{0})\cdot \overrightarrow{1}$ for $f_{n}=P^{n}f-P^{n}f(x_{0})\cdot \overrightarrow{1}$.

In the case of $|a|=+\infty$, for example, $a=+\infty$, define
$Qf\coloneqq\underset{r\to+\infty}{\text{lim}}\tilde{P}(f+r\cdot \overrightarrow{1})-r\cdot \overrightarrow{1}$ satisfying
non-expansion (5), uniform connectedness (7) with $n_{0}=1$ and constant
additivity (4). Then $Q^{k}g$ is linear growth of $k$,
i.e., $Q^{k}g=g+kc\cdot \overrightarrow{1}$. Then the constant $c$ of different accumulation
points are the same by the non-expansion property. And the uniqueness
of accumulation points follows from the uniform connectedness of $Q$.

First, we give a convergence lemma by the uniqueness of accumulation
points. 
\begin{lem}
\label{lem: convergence lemma by the uniqueness}Let $(X,d)$ be a
locally compact metric space. If a sequence $\left\{ x_{n}\right\}\subset X$
has exactly one accumulation point and satisfies 
\[
d(x_{k},x_{k+1})\leq C \quad \text{for some }C>0 \text{ and all } k \in\mathbb{N}_{+},
\]
then $\left\{ x_{n}\right\} $
converges. 
\end{lem}
\begin{proof}
Let $x_{0}$ be the accumulation point. Prove by contradiction. For
some positive $\epsilon$, suppose that there is a subsequence $\left\{ x_{n_{k}}\right\} $
such that $d(x_{n_{k}},x_{0})\geq\epsilon$. For $n_{k},$ let $m_{k}\geq n_{k}$
be the smallest number such that $d(x_{m_{k}+1},x_{0})<\epsilon$,
then $d(x_{m_{k}},x_{0})\in[\epsilon,\epsilon+C]$. Hence $\left\{ x_{m_{k}}\right\} $
has an accumulation point different from $x_{0}$ as the local compactness,
which contradicts to the uniqueness of accumulation points. 
\end{proof}
\begin{rem}
We give a specific example to illustrate that the non-expansion $$d(x_{k},x_{k+1})\leq C$$
is necessary. For $X=\mathbb{R}$, let $x_{2k}=k$, $x_{2k+1}=0$,
which has exactly one accumulation point but does not satisfy $d(x_{k},x_{k+1})\leq C$
for all $k$, and $\left\{ x_{n}\right\} $ does not converge. 
\end{rem}
Next, we prove Theorem \ref{thm:NLMC convergence 3} by the uniqueness
of accumulation points via the uniform connectedness (7). Recall that
if $P$ is uniformly connected, then there exists $n_{0}\in\mathbb{N}_{+}$,
positive $\epsilon_{0}$ such that for every component $x,$ positive $\delta$
and $f,g\in\mathbb{R}^{N}$ with $f\geq g+\delta\cdot1_{x}$, we have
$P^{n_{0}}f\geq P^{n_{0}}g+\epsilon_{0}\delta \cdot \overrightarrow{1}$. 
\begin{proof}[Proof of Theorem \ref{thm:NLMC convergence 3}]
Let $\tilde{P}=P^{n_{0}}$, then $\tilde{P}$ is uniformly connected
with $n_{0}=1$, implying strict monotonicity of corresponding components (2). As $g$ is an accumulation
point, then there exist $0\leq i\leq n_{0}-1$ and a subsequence $\left\{ n_{k}=m_{k}n_{0}+i\right\} $
such that $f_{n_{k}}\to g$ and $P^{n_{k}}f(x_{0})\to a\in[-\infty,+\infty]$
as $k\to+\infty$. Divide it into two cases: $\mid a\mid<+\infty$
and $\mid a\mid=+\infty$.

\textbf{Case 1.} If $\mid a\mid<+\infty$, w.l.o.g, suppose $a>0$.

Then as $k\to+\infty$, $$P^{n_{k}}f=\tilde{P}^{m_{k}}P^{i}f\eqqcolon\tilde{P}^{m_{k}}F\to g+a\cdot \overrightarrow{1}\eqqcolon\tilde{g}.$$
By Theorem \ref{thm:NLMC convergence 2}, we have
$\tilde{P}^{m}F\to\tilde{g}$ as $m\to+\infty$ and $\tilde{P}^{k}\tilde{g}=\tilde{g}$
for all $k\in\mathbb{N}_{+}$. Then for different accumulation points of $f_{n}$,
by the non-expansion property (5), their corresponding $a$ must be
the same case.

If there are two accumulation points $\tilde{g}_{1}\neq\tilde{g}_{2}$,
suppose $\alpha\coloneqq\underset{1\leq x\leq N}{\text{max}}\left(\tilde{g}_{1}(x)-\tilde{g}_{2}(x)\right)$ $>0$.
Then $\tilde{g}_{1}\leq\tilde{g}_{2}+\alpha \cdot \overrightarrow{1}$. If there exists $x$
such that $\tilde{g}_{1}(x)<\tilde{g}_{2}(x)+\alpha$, then by the
uniform connectedness and non-expansion property of $\tilde{P}$, 
\[
\tilde{g}_{1}=\tilde{P}\tilde{g}_{1}<\tilde{P}\left(\tilde{g}_{2}+\alpha \cdot \overrightarrow{1} \right)\leq\tilde{P}\tilde{g}_{2}+\alpha \cdot \overrightarrow{1}=\tilde{g}_{2}+\alpha \cdot \overrightarrow{1},
\]
which implies $\tilde{g}_{1}<\tilde{g}_{2}+\alpha \cdot \overrightarrow{1}$ and contradicts
to the definition of $\alpha$. Hence $\tilde{g}_{1}=\tilde{g}_{2}+\alpha \cdot \overrightarrow{1}$,
which means $\tilde{g}_{1}-\tilde{g}_{1}(x_{0}) \cdot \overrightarrow{1}=\tilde{g}_{2}-\tilde{g}_{2}(x_{0}) \cdot \overrightarrow{1}$.
Then the two accumulation points $g_{i}=\tilde{g}_{i}-\tilde{g}_{i}(x_{0}) \cdot \overrightarrow{1}$
for $i=1,2$ of $f_{n}$ are the same. By Lemma \ref{lem: convergence lemma by the uniqueness},
we get the convergence of $f_{n}$.

\textbf{Case 2.} If $|a|=+\infty$, w.l.o.g., suppose $a=+\infty$.

Since $\tilde{P}(f+r\cdot \overrightarrow{1})-r \cdot \overrightarrow{1}$ is decreasing for positive $r$ by the non-expansion
condition (5), we define 
\[
Qf\coloneqq\underset{r\to+\infty}{\text{lim}}\left(\tilde{P}(f+r\cdot \overrightarrow{1})-r \cdot \overrightarrow{1}\right).
\]
As $$Qg-g=\underset{k\to+\infty}{\text{lim}}\left(\tilde{P}^{m_{k}+1}F-\tilde{P}^{m_{k}}F\right)$$
is finite by the non-expansion property (5) of $\tilde{P}$, then for
any $f\in\mathbb{R}^{N}$, we know $Qf$ is also finite by the non-expansion property.
Then $Q$ satisfies the non-expansion property (5) and constant additivity
$Q(f+c \cdot \overrightarrow{1})=Qf+c \cdot \overrightarrow{1} $.

Now we show the connectedness of $Q$. For some component $x$, positive
$\delta$ and $f,h\in\mathbb{R}^{N}$ with $f\geq h+\delta\cdot1_{x}$,
since $\tilde{P}$ is uniformly connected with $n_{0}=1$, we have
\[
\tilde{P}\left(f+r \cdot \overrightarrow{1}\right)-r \cdot \overrightarrow{1}>\tilde{P}\left(h+r \cdot \overrightarrow{1}\right)-r\cdot \overrightarrow{1}+\epsilon_{0}\delta \cdot \overrightarrow{1}\geq Qh+\epsilon_{0}\delta \cdot \overrightarrow{1}.
\]
Let $r\to+\infty$, then $Qf>Qh+\epsilon_{0}\delta \cdot \overrightarrow{1}$ and $Q$ is
uniformly connected with $n_{0}=1$ and $\epsilon_{0}$.

Let $\lambda_{+}^{\tilde{P}}(f)\coloneqq\underset{1\leq x\leq N}{\text{max}}\left(\tilde{P}f(x)-f(x)\right)$
and $F=P^{i}f$, then $\lambda_{+}^{\tilde{P}}(\tilde{P}^{m}F)>0$
for all $m\in\mathbb{N}_{+}$ since $\tilde{P}^{m_{k}}F(x_{0})\to+\infty$. By the
monotonicity and non-expansion property, 
\[
\tilde{P}^{m+2}F\leq\tilde{P}\left(\tilde{P}^{m}F+\lambda_{+}^{\tilde{P}}(\tilde{P}^{m}F) \cdot \overrightarrow{1}\right)\leq\tilde{P}^{m+1}F+\lambda_{+}^{\tilde{P}}(\tilde{P}^{m}F) \cdot \overrightarrow{1},
\]
which means $\lambda_{+}^{\tilde{P}}(\tilde{P}^{m}F)$ is decreasing
in $m$. Since $g$ is a finite accumulation point, for all $l\in\mathbb{N}_{+}$,
we have 
\[
Qg-g=\underset{k\to+\infty}{\text{lim}}\left(\tilde{P}^{m_{k}+1}F-\tilde{P}^{m_{k}}F\right)
\]
and 
\[
Q^{l+1}g-Q^{l}g=\underset{k\to+\infty}{\text{lim}}\left(\tilde{P}^{m_{k}+l+1}F-\tilde{P}^{m_{k}+l}F\right).
\]
Then by the monotonicity of $\lambda_{+}^{\tilde{P}}(\tilde{P}^{m}F)$,
we have 
\[
\lambda_{+}^{Q}(g)=\underset{1\leq x\leq N}{\text{max}}\left(Qg(x)-g(x)\right)=\underset{m\to+\infty}{\text{lim}}\lambda_{+}^{\tilde{P}}(\tilde{P}^{m}F)=\lambda_{+}^{Q}(Q^{l}g).
\]
Since $Qg\leq g+\lambda_{+}^{Q}(g)\cdot \overrightarrow{1}$, if there is $x$ such that $Qg(x)<g(x)+\lambda_{+}^{Q}(g)$,
then 
\[
Q^{2}g<Qg+\lambda_{+}^{Q}(g)\cdot \overrightarrow{1}=Qg+\lambda_{+}^{Q}(Qg)\cdot \overrightarrow{1}
\]
by the uniform connectedness and constant additivity of $Q$, which
contradicts to the definition of $\lambda_{+}^{Q}(g)$. Hence $Qg=g+\lambda_{+}^{Q}(g)\cdot \overrightarrow{1}$
and $Q^{l}g=g+l\lambda_{+}^{Q}(g)\cdot \overrightarrow{1}$.

Let $f^{(1)},f^{(2)}\in\mathbb{R}^{N}$ and assume $g_{i}$ is an
accumulation point of $f_{n}^{(i)}=P^{n}f^{(i)}-P^{n}f^{(i)}(x_{0})\cdot \overrightarrow{1}$
for $i=1,2$. Then by the non-expansion property 
\[
\left\Vert g_{1}+l\lambda_{+}^{Q}(g_{1})\cdot \overrightarrow{1}-g_{2}-l\lambda_{+}^{Q}(g_{2})\cdot \overrightarrow{1}\right\Vert _{\infty}=\left\Vert Q^{l}g_{1}-Q^{l}g_{2}\right\Vert _{\infty}\leq\left\Vert g_{1}-g_{2}\right\Vert _{\infty},
\]
we know $\lambda_{+}^{Q}(g_{1})=\lambda_{+}^{Q}(g_{2})$. If $g_{1}\neq g_{2}$,
suppose $\alpha\coloneqq\underset{x}{\text{max}}\left(g_{1}(x)-g_{2}(x)\right)>0$.
Then $g_{1}\leq g_{2}+\alpha \cdot \overrightarrow{1}$ and $0=g_{1}(x_{0})<g_{2}(x_{0})+\alpha=\alpha$.
By the uniform connectedness of $Q$, 
\[
g_{1}+\lambda_{+}^{Q}(g_{1})\cdot \overrightarrow{1}=Qg_{1}<Qg_{2}+\alpha \cdot \overrightarrow{1}=g_{2}+\lambda_{+}^{Q}(g_{2}) \cdot \overrightarrow{1}+\alpha \cdot \overrightarrow{1},
\]
which implies $g_{1}<g_{2}+\alpha \cdot \overrightarrow{1}$ and contradicts to the definition
of $\alpha$. Then the two accumulation points are the same.

By Lemma \ref{lem: convergence lemma by the uniqueness} and letting
$f^{(1)}=f^{(2)}=f$, we get the convergence of $f_{n}$. Moreover,
the case $f^{(1)}\neq f^{(2)}$ shows the uniqueness of the limit. 
\end{proof}

\subsection{\label{subsec:An-example-of non-convergence}An example of non-convergence.}

Note that the condition of accumulation points at infinity (8) is
necessary for the convergence. Next, we give a concrete example, which
shows that without the condition of accumulation points (8), then
$P^{n}f(y)-P^{n}f(x)$ does not converge even in $\left[-\infty,\infty\right]$.
First, we prove an extension lemma. 
\begin{lem}
\label{lem:extend P form =00003D00005COmega to R^N}Suppose that a
closed subspace $\Omega\subseteq\mathbb{R}^{N}$ satisfies $a+c \cdot \overrightarrow{1}\in\Omega$
for any $a\in\Omega$ and $c\in\mathbb{R}$. If $P:\Omega\to\Omega$
satisfies uniform strict monotonicity (3) and constant additivity
(4), then $P$ can be extended to $\mathbb{R}^{N}$ with conditions
(3) and (4). 
\end{lem}
\begin{proof}
For any $f\in\mathbb{R}^{N}$, define the extension of $P$ as 
\[
\bar{P}f\coloneqq\text{inf}\left\{ Pg-\epsilon_{0}(g-f):g\in\Omega,g\geq f\right\} ,
\]
where $\epsilon_{0}$ is defined in uniform strict monotonicity (3). Then
$\bar{P}$ satisfies constant additivity (4). If $f_{1}\geq f_{2},$
we know 
\begin{align*}
\bar{P}f_{1} & =\inf\left\{ Pg-\epsilon_{0}(g-f_{1}):g\in\Omega,g\geq f_{1}\right\} \\
 & \geq\inf\left\{ Pg-\epsilon_{0}(g-f_{1}):g\in\Omega,g\geq f_{2}\right\} \\
 & =\bar{P}f_{2}+\epsilon_{0}(f_{1}-f_{2}).
\end{align*}
Then $\bar{P}$ satisfies uniform strict monotonicity (3). Particularly,
$\bar{P}f>-\infty$ for all $f\in\mathbb{R}^{N}$. 
\end{proof}
Now we give a counterexample. Let 
\[
\Omega_{e}\coloneqq\left\{ \left(n,-n,-\varepsilon,\varepsilon\right)+\vec{c}\in\mathbb{R}^{4},n\text{ is even},\vec{c}=(c,c,c,c)\in\mathbb{R}^{4},\varepsilon\ll1\right\} ,
\]
\[
\Omega_{o}\coloneqq\left\{ \left(n,-n,\varepsilon,-\varepsilon\right)+\vec{c}\in\mathbb{R}^{4},n\text{ is odd},\vec{c}=(c,c,c,c)\in\mathbb{R}^{4},\varepsilon\ll1\right\} ,
\]
and define $P:\Omega\coloneqq\Omega_{e}\cup\Omega_{o}\to\Omega$ as
\[
\begin{array}{cc}
P\left(\left(n,-n,-\varepsilon,\varepsilon\right)+\vec{c}\right)=(n+1,-n-1,\varepsilon,-\varepsilon)+\vec{c}, & \text{if }n\text{ is even},\\
P\left(\left(n,-n,\varepsilon,-\varepsilon\right)+\vec{c}\right)=(n+1,-n-1,-\varepsilon,\varepsilon)+\vec{c}, & \text{if }n\text{ is odd}.
\end{array}
\]
Then we want to show that $P$ satisfies uniform strict monotonicity
(3). Suppose 
\[
f=(n,-n,\pm\varepsilon,\mp\varepsilon)+\vec{c}_{1},
\]
\[
g=(m,-m,\pm\varepsilon,\mp\varepsilon)+\overrightarrow{\left(c_{1}-\mid n-m\mid\right)},
\]
then $f\geq g$. And 
\begin{align*}
Pf-Pg-\frac{f-g}{2} & \geq\frac{1}{2}\left[\left(n-m,m-n,-2\varepsilon,-2\varepsilon\right)+\overrightarrow{\mid n-m\mid}\right]\\
 & \geq\left(0,0,0,0\right),
\end{align*}
which means that $P$ satisfies uniform strict monotonicity (3) with
$\epsilon_{0}=\frac{1}{2}$. By Lemma \ref{lem:extend P form =00003D00005COmega to R^N},
one can extend it to $\mathbb{R}^{N}$ with uniform strict monotonicity
(3) and constant additivity (4). But $P^{n}f(x_{3})-P^{n}f(x_{4})$
always jumps between two values $\pm2\varepsilon$ and does not
converge. 
\begin{center}
\includegraphics[scale=0.4]{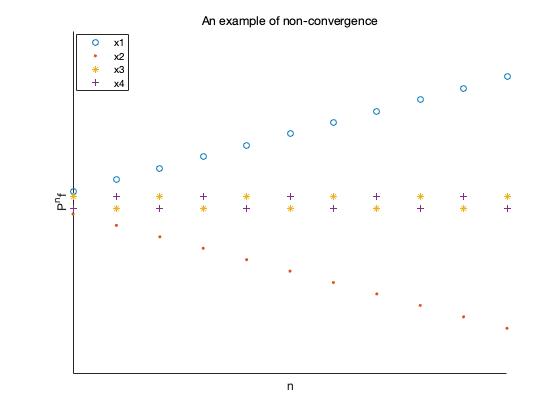} 
\par\end{center}

\begin{center}
\textbf{Figure 1.} This figure shows an example of non-convergence. 
\par\end{center}

\section{Basic facts of graphs}

The main applications of our theorems are parabolic equations on graphs.
Now we give an overview of weighted graphs.

\subsection{Weighted graphs.}

A weighted graph $G=(V,E,w,m,d)$ consists of a countable set $V$, 
a symmetric function $w:V\times V\to[0,+\infty)$ called edge weight
with $w=0$ on the diagonal, and a function $m:V\to\left(0,+\infty\right)$
called vertex weight. The edge weight $w$ induces a symmetric
edge relation $E=\left\{ \left(x,y\right):w(x,y)>0\right\} $. We write $x\sim y$ if $(x,y) \in E$. In the following, we only consider
locally finite graphs, i.e., for every $x\in V$ there are only finitely
many $y\in V$ with $w(x,y)>0$. The degree at $x$ defined as $\text{ Deg}(x)=\underset{y\sim x}{\mathop{\sum}}\frac{w(x,y)}{m(x)}$.
We say a metric $d:V^{2}\rightarrow[0,+\infty)$ is a path metric
on a graph $G$ if 
\[
d(x,y)=\inf\left\{ \stackrel[i=1]{n}{\sum}d(x_{i-1},x_{i}):x=x_{0}\sim\ldots\sim x_{n}=y\right\} .
\]
By assigning each edge of length one, we get the combinatorial distance. The space of all functions defined on the vertex set $V$ is denoted by $\mathbb{R}^V$.
For all $f\in\mathbb{R}^{V},$  define the difference operator for any $x\sim y$
as
\[
\nabla_{xy}f=f(y)-f(x).
\]
The Laplace operator is defined as 
\[
\Delta f(x)\coloneqq\frac{1}{m(x)}\underset{y}{\sum}w(x,y)\nabla_{xy}f.
\]
For $f\in\mathbb{R}^{V}$, we write $\left\Vert f\right\Vert _{\infty}\coloneqq\sup_{x\in V}|f(x)|$
and for $p\geq 1$,
$$\left\Vert f\right\Vert _{p}\coloneqq \left(\underset{x,y}{\sum}\frac{w(x,y)}{m(x)}|f(x)|^{p}\right)^{1/p}.$$

\subsection{Discrete Ricci curvature.}

First, we define the Wasserstein distance on graphs.

\begin{defn}\label{Wasserstein}
    Let \( G=(V,E,w,m,d) \) be a graph, and let \( \nu_1 \) and \( \nu_2 \) be probability measures on \( G \). The Wasserstein distance between \( \nu_1 \) and \( \nu_2 \) is defined as
    \[
    W(\nu_1, \nu_2) := \inf_{\pi} \sum_{x,y \in V} \pi(x,y) d(x,y),
    \]
    where the infimum is taken over all couplings \( \pi: V \times V \to [0,1] \) satisfying:
    \[
    \sum_{y \in V} \pi(x,y) = \nu_1(x) \quad \text{and} \quad \sum_{x \in V} \pi(x,y) = \nu_2(y). 
    \]
\end{defn}

Given a weighted graph \( G=(V,E,w,m,d) \) and a vertex \( x \in V \), we define the probability measure

\begin{equation}\label{mu}
    \mu_x^\varepsilon(z) =
    \begin{cases}
        1 - \varepsilon \text{Deg}(x) & : z = x, \\
        \varepsilon w(x,z)/m(x)       & : z \sim x, \\
        0                             & : \text{otherwise},
    \end{cases}
\end{equation}
where \( 0 \le \varepsilon \le 1/\text{Deg}(x) \).

Next, we introduce the Lin-Lu-Yau-Ollivier and Ollivier Ricci curvature.

\begin{defn}\label{def-curvature}
    For a locally finite weighted graph \( G = (V, E, w, m,d) \), the \( \varepsilon \)-Ollivier Ricci curvature between vertices \( x \neq y \) is defined as
    \[
    \kappa_{\varepsilon}(x,y) := 1 - \frac{W(\mu_x^{\varepsilon}, \mu_y^{\varepsilon})}{d(x,y)}.
    \]
    The Lin-Lu-Yau-Ollivier Ricci curvature between vertices \( x \neq y \) is given by
    \begin{equation}\label{lim-LLY_curvature}
        \kappa_{LLY}(x,y) := \lim_{\varepsilon \to 0^+} \frac{1}{\varepsilon} \kappa_{\varepsilon}(x,y).
    \end{equation}
    In particular, for weighted graphs \( G \) with \( \text{Deg}(x) \le 1 \) for all \( x \in V \) (including normalized graphs), the Ollivier Ricci curvature is defined as
    \begin{equation}\label{Ollivier_curvature}
        \kappa(x,y) := \kappa_1(x,y).
    \end{equation}
\end{defn}

\begin{rem}
The limit expression (\ref{lim-LLY_curvature}) for the Lin-Lu-Yau-Ollivier Ricci curvature is well-defined due to the work of \cite{MW19}, which showed that \( \kappa_{\varepsilon}\) is a piecewise linear concave function with at most three linear parts. This ensures the existence of the limit. Furthermore, the authors \cite{MW19} derived two equivalent limit-free expressions for (\ref{lim-LLY_curvature}).
\end{rem}






\subsection{Nonlinear Laplace and resolvent operators.}

On a locally finite weighted graph $G=(V,E,w,m,d)$, for every $p\geq1$,
define the energy functional of $f\in\mathbb{R}^{V}$ as 
\[
\mathscr{E}_{p}(f)=\frac{1}{2}\underset{x,y\in V}{\sum}\frac{w(x,y)}{m(x)}\lvert \nabla_{xy}f \rvert^{p},
\]
where $\nabla_{xy}f=f(y)-f(x)$.
More explicitly, the $p$-Laplace operator $\Delta_{p}:\mathbb{R}^{V}\to\mathbb{R}^{V}$
is given by 
\[
\Delta_{p}f(x)\coloneqq\frac{1}{m(x)}\underset{y}{\sum}w(x,y)\lvert \nabla_{xy}f\rvert^{p-2} \nabla_{xy}f,\text{if }p>1,
\]
and 
\[
\Delta_{1}f(x)\in \frac{1}{m(x)}\underset{y}{\sum}w(x,y)\,\textrm{sign}(\nabla_{xy}f) \text{, }
\]
where $\text{sign}\ensuremath{\left(t\right)}=\ensuremath{\begin{cases}
\begin{array}{c}
1,\\{}
[-1,1]\\
-1,
\end{array}, & \begin{array}{c}
t>0.\\
t=0.\\
t<0.
\end{array}\end{cases}}$ Note that $p=2$ is the general discrete Laplace operator $\Delta$.

The resolvent operator of $\Delta_{p}$ is defined as $J_{\epsilon}=\left(id-\epsilon\Delta_{p}\right)^{-1}$
for $\epsilon>0$. Since $-\Delta_{p}$ is a monotone operator, i.e.,
for all $f,g\in\mathbb{R}^{V}$, we have 
\[
\left\langle -\Delta_{p}f+\Delta_{p}g,f-g\right\rangle \geq0.
\]
Then the resolvent $J_{\epsilon}$ is single-valued, monotone, and
non-expansive, i.e., 
\[
\parallel J_{\epsilon}f-J_{\epsilon}g\parallel_{\infty}\leq\parallel f-g\parallel_{\infty}.
\]
See details in \cite[Corollary 2.10]{M92} \cite[Proposition 12.19]{RW09}.
Moreover, since $\Delta_{p}$ is the subdifferential of convex functional
$\mathscr{E}_{p}$, it follows that 
\[
J_{\epsilon}f=\underset{g\in\mathbb{R}^{V}}{\text{argmin}}\left\{ \mathscr{E}_{p}(g)+\frac{1}{2\epsilon}\left\Vert g-f\right\Vert _{2}^{2}\right\} .
\]



\section{Applications}

In this section, we prove that our convergence and uniqueness results
have important applications in the Ollivier Ricci curvature flow,
the Laplacian separation flow, the nonlinear Dirichlet form, and the nonlinear
Perron-Frobenius theory.

\subsection{The convergence of Ollivier Ricci curvature flow.}

Consider a finite weighted graph $G=(V,E,w,m,d)$. For an initial metric $d_{0}$, fix some $C$ as the deletion threshold such that $C>\underset{x\sim y\sim z}{\max}\frac{d_{0}(x,y)}{d_{0}(y,z)}$. Then we can execute the following algorithm.

\begin{table}[ht]
\centering
\captionsetup{name=Algorithm} 
\begin{tabular}{p{\textwidth}}
\toprule
\textbf{Algorithm 1: Discrete Ricci flow with surgery}  \\%
\midrule
\textbf{Input:} %
Weighted graph \( G \), initial metric \( d_0 \), deletion threshold \( C \), iteration rate \( 0 < \alpha < 1 \), precision \( \delta \);\\
\textbf{Output:} %
Updated graph \( \tilde{G} \) (with connected components \( \tilde{G}_i \)), corresponding metrics \( d^i \), constant curvatures \( \kappa^i \). \\%

\begin{itemize}
  \item[1:] For each edge \( x \sim y \in E \), update the metric via the discrete Ricci flow
  \begin{equation}\label{eq: LLL curvature flow}
  d_{n+1}(x,y) \leftarrow d_{n}(x,y) - \alpha \kappa_{d_{n}}(x,y) d_{n}(x,y).
  \end{equation}
  \item[2:] If adjacent edges \( x \sim y \sim z \) satisfy \( d_{n+1}(x,y) > C d_{n+1}(y,z) \), delete the edge \( x \sim y \) from \( E \). The deletion process proceeds in non-increasing order of edge length, with ties broken by the order of appearance. Denote the updated graph by \( \tilde{G} \).
  \item[3:] On each connected component of \( \tilde{G} \), update the distance between non-adjacent vertices \( x \nsim y \) by
  \[
  d_{n+1}(x,y) \leftarrow \inf \left\{ \sum_{i=1}^{k} d_{n+1}(x_{i-1},x_{i}) : x = x_{0} \sim x_{1} \cdots \sim x_{k} = y \right\}.
  \]
  \item[4:] Repeat steps 1-3 until the precision condition \( \lVert d_{n+1}^i - d_{n}^i \rVert_{\ell^{\infty}} < \delta \) is met on every connected component \( \tilde{G}_i \), then compute the corresponding constant curvature \( \kappa^i \).
\end{itemize} \\%
\bottomrule
\end{tabular}
\caption{Discrete Ricci flow with surgery algorithm.}  
\label{tbl:algorithm1}  
\end{table}


Since the graph $G$ is finite and graphs of a single edge can not
be deleted, denote the new graph as $\tilde{G}$ after the last edge
deletion. Then on each connected component of $\tilde{G}$, the distance
ratios are bounded in $n$, and hence, $\text{log}d_{n}$ has an accumulation
point at infinity. Considering Ricci flow (\ref{eq: LLL curvature flow})
as a nonlinear Markov chain on each connected component of $\tilde{G}$,
by Theorem \ref{thm:NLMC convergence 1} we can prove that (\ref{eq: LLL curvature flow})
converges to a constant curvature metric.


\noindent \textbf{Theorem 4.} \textit{Let $d_0$ be an initial metric on a finite weighted graph $G = (V, E, w, m, d_0)$ with $\text{Deg}(x)\leq1$ for
all $x\in V$. Through the discrete Ricci flow with surgery (Algorithm \ref{tbl:algorithm1}), $\frac{d_n(e)}{\text{max } d_n(e')}$  converges to a constant-curvature metric on each connected component of the final graph $\tilde{G}$, where the max is taken over all $ e'$ in the same connected component as $e$ on $\tilde{G}$.}
\begin{proof}
For $f\in\mathbb{R}_{+}^{E}$, define $Sf\in\mathbb{R}^{V\times V}$
as 
\[
Sf(x,y)=\underset{}{\text{inf}}\left\{ \stackrel[i=1]{k}{\sum}f(x_{i-1},x_{i}):x=x_{0}\sim x_{1}\cdots\sim x_{k}=y\right\} .
\]
Note that $Sf$ is a distance function on $G$. By (\ref{eq: LLL curvature flow}),
for $f\in\mathbb{R}_{+}^{E}$, define 
\[
\tilde{P}f(x,y)\coloneqq\alpha W_{Sf}(\mu^1_x,\mu^1_y)+(1-\alpha)f(x,y),
\]
where $W_{Sf}$ is the Wasserstein distance corresponding to the distance
$Sf$, see Definition \ref{Wasserstein}. Then $\tilde{P}$ corresponds
to step 1 in Algorithm \ref{tbl:algorithm1} of the Ollivier Ricci flow, that is, 
\[
d_{n+1}\mid_{E}=\tilde{P}(d_{n}\mid_{E}).
\]
Clearly, $\tilde{P}$ satisfies monotonicity (1) and strict monotonicity of corresponding components
(2). Since 
\[
\tilde{P}(rd)=r\tilde{P}d,\:\forall r>0,
\]
define $Pf\coloneqq\text{log \ensuremath{\tilde{P}}(exp\ensuremath{(f)})}$
with $f=\text{log}\,d$. Then for every constant $c\in\mathbb{R}$,
\[
P(f+c\cdot \overrightarrow{1})=\text{log}(\tilde{P}(\exp f\cdot\exp c ))=\text{log}(\exp c\,\tilde{P}(\exp f))=c\cdot \overrightarrow{1}+Pf,
\]
which implies that $P$ satisfies constant additivity (4). And $P$
also satisfies monotonicity (1) and strict monotonicity of corresponding components (2). After Algorithm \ref{tbl:algorithm1}, the deletion process (steps 2 and 3) ensures that, on every connected component of the final graph $\tilde{G}$ containing $ e'$, the ratio
$\frac{d_{n}}{d_{n}( e')}$ has a finite positive accumulation point
$d$. Hence, $g=\text{log}\,d$ is a finite accumulation point of
$P^{n}f-P^{n}f(e')\cdot \overrightarrow{1}$. Then by Theorem \ref{thm:NLMC convergence 1},
we know that $P^{n}f-P^{n}f( e')\cdot \overrightarrow{1}$ converges to $g$. Moreover,
we know $Pg=g+c\cdot \overrightarrow{1}$ and $\tilde{P}d=\tilde{c}d$. That is, its curvature
is a constant. 
\end{proof}
Take a simple example to illustrate Algorithm \ref{tbl:algorithm1}. Let $G=(V,E,w,m,d)$ be a normalized graph with unit edge weights, where $V=\{x_i\}_{i=0}^5$, $E=\{x_0x_1,x_0x_2,x_1x_2,$ $x_2x_3,x_3x_4,x_3x_5,x_4x_5\}$, $w\equiv1$, $m(x)=|\{y\in V:y \sim x\}|$ and $d$ is the combinatorial distance
\[
d(x,y)=\inf\left\{ n:x=x_{0}\sim\ldots\sim x_{n}=y\right\} .
\]
The length of each edge and its corresponding Ollivier curvature are indicated in Figure 2.
\begin{center}
\includegraphics[scale=0.33]{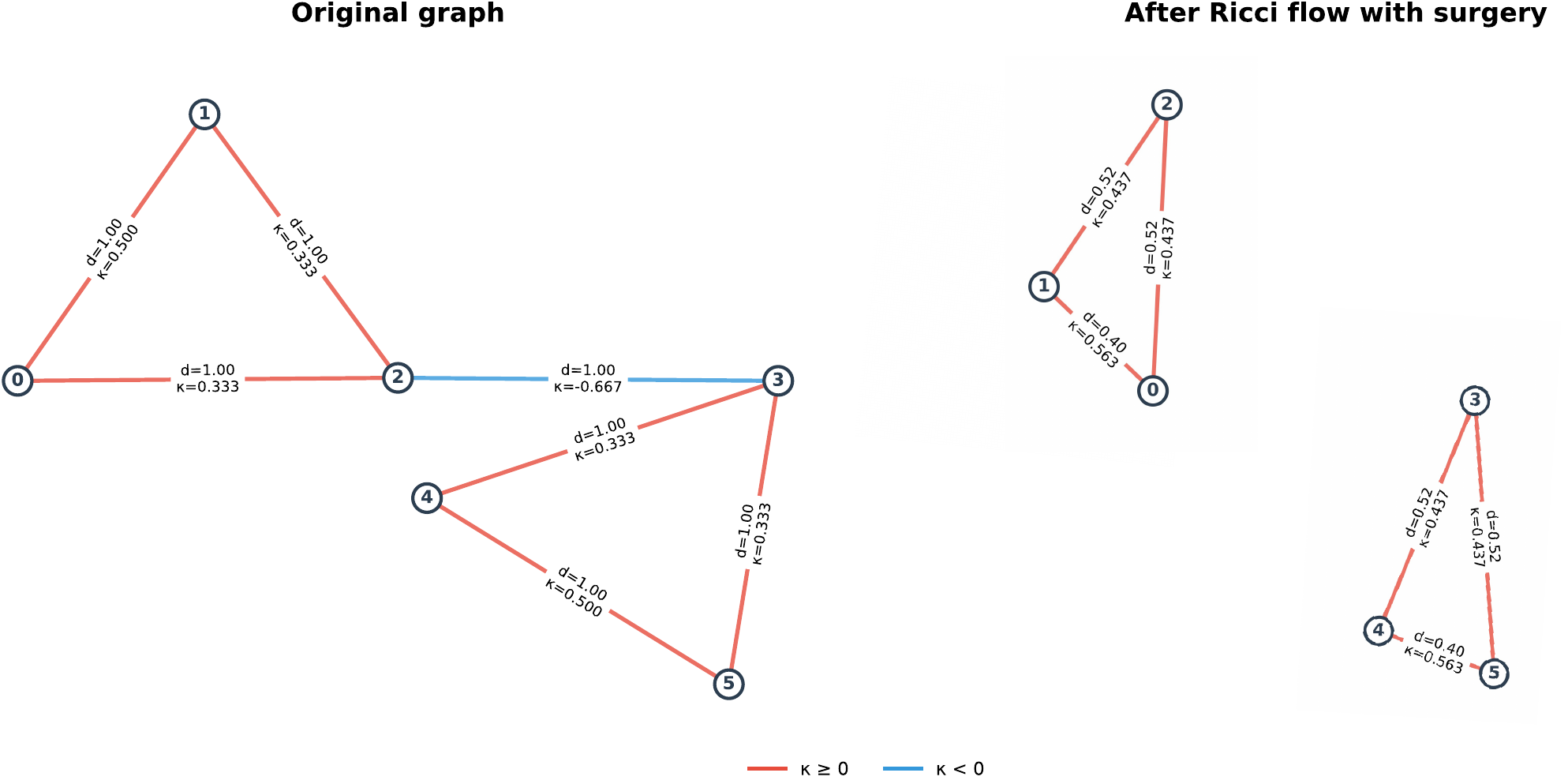} 
\par\end{center}

\begin{center}
\textbf{Figure 2.} An example of Algorithm 1. 
\par\end{center}

Setting $\alpha=0.7$, through the Ricci flow (\ref{eq: LLL curvature flow}), all edges with positive curvature are shortened, while the edge of negative curvature $x_2x_3$ is elongated. Once its length exceeds the prescribed threshold $C=1.5$, it is removed. The resulting graph consists of two 3-cycles (see the right panel of Figure 2).

\subsection{The gradient estimate for resolvents of nonlinear Laplace.}

It is shown in \cite{MW19} that a lower Ollivier curvature bound
is equivalent to a gradient estimate for the continuous time heat
equation. In \cite{LL18,M18,m19li,dier2021discrete,weber2023li,horn2019volume}, the gradient estimates have been proved
under Bakry-Emery curvature bounds. In \cite{CMS24}, the authors
proved that the nonnegative sectional curvature implies a logarithmic
gradient estimate. Gradient estimates of the discrete random walk $id+\varepsilon\Delta$  have been proved in \cite{BCMP18,LLY11,HM21}. In \cite[Theorem 5.2.]{IKTU21},
the authors showed a gradient estimate for the coarse Ricci curvature
defined on hypergraphs.

Here we modify the definition of Ollivier Ricci curvature and prove
the Lipschitz decay for nonlinear parabolic equations. On a locally
finite weighted graph $G=(V,E,w,m,d_{0})$ with the combinatorial distance
$d_{0}$, for all $f\in\mathbb{R}^{V}$ define 
\[
\Delta_{\phi}f(x)\coloneqq\underset{}{\underset{y}{\sum}\frac{w(x,y)}{m(x)}\phi(f(y)-f(x)),}
\]
where $\phi:\mathbb{R}\rightarrow\mathbb{R}$, is odd, increasing,
and either convex or concave on $\mathbb{R}_{+}$. Recall the transport
plan set for $x\neq y\in V$
\[
\Pi\coloneqq\left\{ \pi:B_{1}(x)\times B_{1}(y)\to[0,\infty):\begin{array}{c}
	\underset{x'\in B_{1}(x)}{\sum}\pi(x',y')=\frac{w(y,y')}{m(y)}\text{ for all }y'\sim y,\\
	\underset{y'\in B_{1}(y)}{\sum}\pi(x',y')=\frac{w(x,x')}{m(x)}\text{ for all }x'\sim x,
\end{array}\right\},
\]
where $B_{1}(x)=\{x'|x'\sim x\}\cup \{x\}$.
Then modify the curvature as 
\begin{equation}
	\hat{k}_{\phi}(x,y)\coloneqq\underset{\pi_{\phi}\in\Pi_{\phi}}{\text{sup}}\underset{x',y'\in B_{1}(x)\times B_{1}(y)}{\sum}\pi_{\phi}(x',y')\left(1-\frac{d_{0}(x',y')}{d_{0}(x,y)}\right),\forall x\neq y\in V,\label{eq:new curvature}
\end{equation}
where 
\[
\Pi_{\phi}\coloneqq\left\{ \pi_{\phi}\in\Pi:\begin{array}{c}
	\pi_{\phi}(x',y')=0\text{ if }x'=y'\text{ for convex \ensuremath{\phi}},\\
	\pi_{\phi}(x',y')=0\text{ if }x'\neq x,y'\neq y\text{ and }d_{0}(x',y')=2\text{ for concave \ensuremath{\phi}.}
\end{array}\right\} .
\]

Then we give the gradient estimate for resolvents of nonlinear Laplace. 
\begin{thm}
	\label{thm: Lipschitz decay}Let $G=(V,E,w,m,d_{0})$ be a locally finite weighted graph with combinatorial distance $d_{0}$. If the modified curvature defined in (\ref{eq:new curvature}) has
	a lower bound $\inf_{x\neq y \in V}\hat{k}_{\phi}(x,y)\geq K\geq 0$, then for any $f\in \mathbb{R}^V$ with $\text{Lip}(f):=\sup_{\substack{x\neq y\in V}}\frac{|f(x)-f(y)|}{d_0(x,y)}$ $>0$
	the resolvent $J_{\epsilon}=\left(id-\epsilon\Delta_{\phi}\right)^{-1}$
	satisfies the Lipschitz decay 
	\[
	Lip(J_{\epsilon}f)\leq Lip(f)\left(1+\epsilon\left(Lip(f)\right)^{-1}\phi(Lip(f))K\right)^{-1},
	\]
	and $Lip(J_{\epsilon}f)=0$ for $\text{Lip}(f)=0$. 
\end{thm}
\begin{proof}
	For any $f\in \mathbb{R}^V$ and $x\sim y \in E$, suppose $Lip(f)=C$, $f(y)=C$ and $f(x)=0$,
	then for any $\pi_{\phi}(x',y')$ satisfying the conditions in (\ref{eq:new curvature}),
	\[
	\Delta_{\phi}f(x)-\Delta_{\phi}f(y)=\underset{x',y'}{\sum}\pi_{\phi}(x',y')\left[\phi(f(x')-f(x))-\phi(f(y')-f(y))\right].
	\]
	If $d_{0}(x',y')=1$, then 
	\[
	f(x')-f(x)\geq f(y')-C-(f(y)-C)=f(y')-f(y).
	\]
	Since $\phi$ is increasing, we know 
	\[
	\phi(f(x')-f(x))-\phi(f(y')-f(y))\geq0=d_{0}(x,y)-d_{0}(x',y').
	\]
	If $d_{0}(x',y')=2$, and $x'\neq x$ and $y'\neq y$, then $f(x')-f(x)\geq-C$
	and $f(y')-f(y)\leq C$, and $f(y')-f(x')\leq C$. For convex $\phi$,
	such as $p$-Laplace ($p\geq2$), we have 
	\[
	\phi(f(x')-f(x))-\phi(f(y')-f(y))\geq-\phi(C).
	\]
	If $d_{0}(x',y')=2$ and either $x'=x$ or $y'=y$, then 
	\[
	\phi(f(x')-f(x))-\phi(f(y')-f(y))\geq-\phi(C).
	\]
	If $d_{0}(x',y')=0$, then $0\leq f(x')=f(y')\leq C$. And for concave
	$\phi$, such as $p$-Laplace ($1<p<2$), we have 
	\[
	\phi(f(x')-f(x))-\phi(f(y')-f(y))=\phi(f(x'))-\phi(f(x')-C)\geq\phi(C).
	\]
	Hence, 
	\begin{align*}
		& \Delta_{\phi}f(x)-\Delta_{\phi}f(y)\\
		= & \underset{x',y'}{\sum}\pi_{\phi}(x',y')\left[\phi(f(x')-f(x))-\phi(f(y')-f(y))\right]\\
		= & \left(\underset{d(x',y')=1}{\sum}+\underset{d(x',y')=2}{\sum}+\underset{d(x',y')=0}{\sum}\right)\pi_{\phi}(x',y')\left[\phi(f(x')-f(x))-\phi(f(y')-f(y))\right]\\
		\geq & \phi(C)\underset{x',y'}{\sum}\pi_{\phi}(x',y')\left[d_{0}(x,y)-d_{0}(x',y')\right].
	\end{align*}
	That is, 
	\[
	\Delta_{\phi}f(x)-\Delta_{\phi}f(y)\geq\phi(C)\hat{k}_{\phi}(x,y)d_{0}(x,y)\geq\phi(C)K.
	\]
	Then 
	\[
	(id-\epsilon\Delta_{\phi})f(y)-(id-\epsilon\Delta_{\phi})f(x)\geq C+\epsilon\phi(C)K.
	\]
	For $g\in\mathbb{R}^{V}$, since $Lip(f)=|\nabla f|_{\infty}:=\underset{x\sim y}{\text{sup}}|\nabla_{xy}f|$,
	then by the definition of $J_{\epsilon}$, 
	\begin{align*}
		\underset{|\nabla g|_{\infty}\leq c}{\text{sup}}|\nabla J_{\epsilon}g|_{\infty} & =\underset{|\nabla g|_{\infty}\leq c}{\text{sup}}|\nabla\left(id-\epsilon\Delta_{\phi}\right)^{-1}g|_{\infty}\\
		& =\underset{|\nabla\left(id-\epsilon\Delta_{\phi}\right)h|_{\infty}\leq c}{\text{sup}}|\nabla h|_{\infty}\\
		& =\left(\underset{|\nabla h|_{\infty}\geq c^{-1}}{\text{inf}}|\nabla(id-\epsilon\Delta_{\phi})h|_{\infty}\right)^{-1}.
	\end{align*}
	Thus, we know $$Lip(J_{\epsilon}f)\leq C^{2}\left(C+\epsilon\phi(C)K\right)^{-1}=C\left(1+\epsilon C^{-1}\phi(C)K\right)^{-1}.$$ 
\end{proof}
\begin{rem}
	Since $\Delta_{1}$ is a set-valued function, we cannot directly apply Theorem \ref{thm: Lipschitz decay} to obtain the Lipschitz decay property. The energy functional $\mathscr{E}_{p}(f)$ is uniformly continuous
	with respect to $p$, which means that for any $\delta>0$, there
	exists $p$ only depending on $\delta$ such that for all $f_{0}\in\mathbb{R}^{V}$,
	\[
	\underset{f:\left\Vert f-f_{0}\right\Vert {}_{\infty}\leq1}{\sup}|\mathscr{E}_{p}(f)-\mathscr{E}_{1}(f)|\leq\delta.
	\]
	For fixed $f\in\mathbb{R}^{V}$ and $\epsilon>0$, the resolvent $J_{\epsilon}^{p}f=\underset{g\in\mathbb{R}^{V}}{\text{argmin}}\left\{ \mathscr{E}_{p}(g)+\frac{1}{2\epsilon}\left\Vert g-f\right\Vert _{2}^{2}\right\} $
	is also continuous with respect to $p$. Hence, by the Lipschitz decay
	property of $J_{\epsilon}^{p}$ for $p>1$, we can deduce the Lipschitz
	decay for $J_{\epsilon}^{1}$. 
\end{rem}

\subsection{The convergence of Laplacian separation flow.}

\noindent Recall that the extremal 1-Lipschitz extension operator
$S$ is defined as $S:\mathbb{R}^{K}\to\mathbb{R}^{V}$, 
\[
Sf(x)\coloneqq\begin{cases}
\begin{array}{c}
f(x):\\
\underset{y\in K}{\text{min}}\left(f(y)+d(x,y)\right):\\
\underset{y\in K}{\text{max}}\left(f(y)-d(x,y)\right):
\end{array} & \begin{array}{c}
x\in K,\\
x\in Y,\\
x\in X,
\end{array}\end{cases}
\]
where $d:V^2\to \mathbb{R}_+$ is a graph distance function on $G$. Then $S(Lip(1,K))\subseteq Lip(1,V)$.
In \cite{HM21}, it is proven via elliptic methods that there exists
some $g$ with $\Delta Sg=\text{const}$. Here we give the parabolic
flow $(id+\epsilon\Delta)S$, and show that it converges to the
constant Laplacian solution, assuming nonnegative Ollivier Ricci
curvature.

\noindent \textbf{Theorem 5.} \textit{Let $G$ be a locally finite graph with nonnegative Ollivier curvature, and let $x_0 \in K$. Define 
$P \coloneqq \left((id + \epsilon \Delta) S\right)\big|_K$,
where $\epsilon > 0$ is sufficiently small so that $diag(id + \epsilon \Delta)$ is positive on $C_0(\bar{K})$. Then for any $f \in Lip(1, K)$, there exists $g \in Lip(1, K)$ such that
\[
P^n f - P^n f(x_0) \cdot \overrightarrow{1} \to g,
\]
and
\[
\left.\Delta Sg\right|_X \geq \left.\Delta Sg\right|_K \equiv \text{const} \geq \left.\Delta Sg\right|_Y.
\] } 
\begin{proof}
We can check that $P$ satisfies monotonicity (1), strict monotonicity of corresponding components
(2), and constant additivity (4). Since the nonnegative Ollivier
Ricci curvature implies Lipschitz decay property of $P$, i.e., the
range of $P$ is $Lip(1,K)$, then there is a finite accumulation
point $g$ of $f_{n}=P^{n}f-P^{n}f(x_{0})\cdot \overrightarrow{1}$. By Theorem \ref{thm:NLMC convergence 1},
we can get the convergence and $g$ is a stationary point. Then $\left.\Delta Sg\right|_X \geq \left.\Delta Sg\right|_K \equiv \text{const} \geq \left.\Delta Sg\right|_Y$ since the nonnegative Ollivier
Ricci curvature.  
\end{proof}
Next, we aim to generalize the result to nonlinear cases. For $p\geq1$,
the resolvent operator of $p$-Laplace operator $\Delta_{p}$ is defined
as $J_{\epsilon}=\left(id-\epsilon\Delta_{p}\right)^{-1}$ for $\epsilon>0$.
Then $J_{\epsilon}$ is monotone \cite[Proposition 12.19]{RW09}.
Moreover, the following lemma asserts that the resolvent $J_{\epsilon}$
satisfies the strict monotonicity of corresponding components property (2). 
\begin{lem}
\label{lem:strict monotonicity}If $f\geq g+\delta|V|1_{x}$,
where $1_{x}(x)=1$ and $1_{x}(y)=0$ for $y\neq x$, then $J_{\epsilon}f(x)\geq J_{\epsilon}g(x)+\delta.$ 
\end{lem}
\begin{proof}
Since $J_{\epsilon}$ is monotone, which means $\left\langle J_{\epsilon}f-J_{\epsilon}g,f-g\right\rangle \geq0.$
Take $f=g+\delta\left(|V|1_{x}- \overrightarrow{1}\right)$, then $\left\langle f-g,\overrightarrow{1}\right\rangle =0$.
By the monotone property, 
\[
\left\langle J_{\epsilon}f-J_{\epsilon}g,\delta\left(|V|1_{x}-\overrightarrow{1}\right)\right\rangle \geq0,
\]
which implies $|V|\left(J_{\epsilon}f-J_{\epsilon}g\right)(x)\geq\left\langle J_{\epsilon}f-J_{\epsilon}g, \overrightarrow{1}\right\rangle =0.$
And set $\tilde{f}=\delta \cdot \overrightarrow{1}+f=g+\delta|V|1_{x}$ which satisfies
$\tilde{f}\geq g+\delta|V|1_{x}$, then 
\[
J_{\epsilon}\tilde{f}(x)=\delta+J_{\epsilon}f(x)\geq\delta+J_{\epsilon}g(x).
\]
This finishes the proof.
\end{proof}
Recall that a new curvature $\hat{k}_{\phi}(x,y)$ is defined in (\ref{eq:new curvature}),
whose transport plans forbid 3-cycles for convex $\phi$ on $\mathbb{R}_{+}$
and forbid 5-cycles for concave $\phi$ on $\mathbb{R}_{+}$.


\noindent \textbf{Theorem 6.} \textit{Let $G$ be a locally finite graph with a nonnegative modified curvature $\hat{k}$, and let $x_0 \in K$. Define $P:=\left((id+\epsilon\Delta_{p})S\right)\mid_{K}$, where $\epsilon > 0$ is sufficiently small so that $\text{diag}(id+\epsilon\Delta_{p})$
is positive on $C_0(\bar{K})$. Then for all $f\in Lip(1,K)$, there exists $\tilde{f}\in Lip(1,K)$
such that $$P^{n}f-P^{n}f(x_{0}) \cdot \overrightarrow{1}\to \tilde{f}.$$ Moreover,
there exist $h,g\in\mathbb{R}^{V}$ such that $g\in\Delta_{p}Sh$
and $g\mid_{X}\geq g\mid_{K}\equiv\text{const}\geq g\mid_{Y}$, where $Sh:=S(h|_K)$.}
\begin{proof}
By Lemma \ref{lem:strict monotonicity} we know $J_{\epsilon}$ satisfies
strict monotonicity of corresponding components property. It is also constant additive by the
definition of resolvent $J_{\epsilon}$, then \(P:=J_{\epsilon}S\mid_{K}\) also satisfies
the same property. For the nonnegative curvature $\hat{k}_{\phi}$
defined as (\ref{eq:new curvature}), by the gradient estimate of
Theorem \ref{thm: Lipschitz decay}, the range of $P$ still is $Lip(1,K)$.
Then there is an accumulation point at infinity. Hence by Theorem
\ref{thm:NLMC convergence 1}, there exists $\tilde{f}\in Lip(1,K)$
such that $P\tilde{f}=\tilde{f}+\text{const}\cdot \overrightarrow{1}$, which implies that
\begin{equation}
SJ_{\epsilon}S\tilde{f}=S\tilde{f}+S(\text{const} \cdot \overrightarrow{1}).\label{eq:the result by the thm}
\end{equation}
Define $h_{\epsilon}\coloneqq J_{\epsilon}S\tilde{f}$ and substitute
it into the above formula (\ref{eq:the result by the thm}), then
we can get \(Sh_{\epsilon}-h_{\epsilon}+\epsilon\Delta_{p}h_{\epsilon}=S(\text{const}\cdot \overrightarrow{1})\). Note that $h_{\epsilon}\neq Sh_{\epsilon}$,
but we claim that $\left\Vert h_{\epsilon}-Sh_{\epsilon}\right\Vert _{\infty}\leq c\epsilon$
for some constant $c$. Since 
\[
\left\Vert h_{\epsilon}-S\tilde{f}\right\Vert _{\infty}=\left\Vert J_{\epsilon}S\tilde{f}-S\tilde{f}\right\Vert _{\infty}\leq\epsilon\left\Vert \Delta_{p}S\tilde{f}\right\Vert _{\infty}\leq\epsilon\,\underset{x}{\max}\,\text{Deg}(x),
\]
and it also holds on $K$, that is, 
\[
\left\Vert Sh_{\epsilon}-SS\tilde{f}\right\Vert _{\infty}=\left\Vert Sh_{\epsilon}-S\tilde{f}\right\Vert _{\infty}\leq\epsilon.
\]
So by the triangle inequality, we get $\left\Vert h_{\epsilon}-Sh_{\epsilon}\right\Vert _{\infty}\leq c\epsilon.$
Then we get $h_{\epsilon}\to h$ for some subsequence as $\epsilon\to0$ and $h=Sh$
by the compactness. Take a subsequence $\left\{ g_{\epsilon}\right\} $
such that $g_{\epsilon}\in\Delta_{p}h_{\epsilon}$ and $g_{\epsilon}|_{X}\geq g_{\epsilon}|_{K}\equiv\text{const}\geq g_{\epsilon}|_{Y}$
and $g_{\epsilon}\to g.$ By the continuity, we know $g\in\Delta_{p}Sh$
and $g|_{X}\geq g|_{K}\equiv\text{const}\geq g\mid_{Y}$. 
\end{proof}

\subsection{The nonlinear Dirichlet form\label{subsec:The-nonlinear-Dirichlet}.}

In the theory of nonlinear Dirichlet form, one has a correspondence
between such forms, semigroups, resolvents, and operators satisfying
suitable conditions. Since the assumptions of our theorems fit well
in the nonlinear Dirichlet form theory, we can apply our theorems
to study the long-time behavior of associated continuous semigroups.
First, we recall the definition of the nonlinear Dirichlet form \cite{CB21}. 
\begin{defn}\label{Def:Dirichletform}
Let $\mathscr{E}:$ $\mathbb{R}^{N}\to[0,\infty]$ be a convex and
lower semicontinuous functional with dense effective domain. Then
the subgradient $-\partial\mathscr{E}$ generates a strongly continuous
contraction semigroup $T$, that is, $u(t)=T_{t}u_{0}$ satisfies
\[
\begin{cases}
\begin{array}{c}
0\in\frac{du}{dt}(t)+\partial\mathscr{E}(u(t)),\\
u(0)=u_{0}
\end{array}\end{cases}
\]
pointwise for almost all $t\geq0$. We call $\mathscr{E}$ a Dirichlet
form if the associated strongly continuous contraction semigroup $T$
is sub-Markovian, which means $T$ is order-preserving and $L^{\infty}$
contractive, that is, for all $u,v\in\mathbb{R}^{N}$ and all
$t\geq0$, 
\[
u\leq v\Rightarrow T_{t}u\leq T_{t}v
\]
and 
\[
\left\Vert T_{t}u-T_{t}v\right\Vert _{\infty}\leq\left\Vert u-v\right\Vert _{\infty}.
\]
\end{defn}
For the definition of the nonlinear Dirichlet form in \cite{J98},
it satisfies the following lemma. 
\begin{lem}
\cite[Lemma 1.1]{J98} For a nonlinear Dirichlet form $\mathscr{E}$,
if $f\in\mathbb{R}^{N}$ and $a,\lambda\in\mathbb{R}$, then 
\begin{equation}
\begin{array}{c}
\mathscr{E}(\lambda f)=\lambda^{2}\mathscr{E}(f),\\
\mathscr{E}(f+a\cdot \overrightarrow{1})=\mathscr{E}(f).
\end{array}\label{eq: property constant additivity}
\end{equation}
\end{lem}
Then we can apply Theorem \ref{thm:NLMC convergence 1} to obtain
the following convergence result with the accumulation point assumption. 
\begin{thm}
\label{thm: nonlinear Dirichlet form}For a nonlinear Dirichlet form
$\mathscr{E}$ with property (\ref{eq: property constant additivity}),
if the semigroup $T_{t}^{n}f$ defined in Definition \ref{Def:Dirichletform} has an accumulation point at infinity
(8) for $t>0$, and its associated generator $-\partial\mathscr{E}$
is bounded, then $T_{t}^{n}f$ converges. 
\end{thm}
\begin{proof}
Define the Markov chain as $P\coloneqq T_{t}$. By the definition
of a Dirichlet form, we know $P$ satisfying monotonicity (1) and
non-expansion (5). And the property (\ref{eq: property constant additivity})
induces the constant additivity (4) of $P$. Since the associated
generator $-\partial\mathscr{E}$ is bounded, then $P$ satisfies strict monotonicity of corresponding components (2).
By the assumption of accumulation points at infinity, we can
get the convergence result by Theorem \ref{thm:NLMC convergence 1}. 
\end{proof}
\begin{rem}
(a) In the nonlinear Dirichlet form setting of \cite{J98}, if we
assume the associated semigroup satisfying sub-Markovian property,
then we can also apply Theorem \ref{thm:NLMC convergence 1} to obtain
convergence results with the accumulation point at infinity assumption
(8).

(b) As we mentioned before, we can study the long-time behavior of
the resolvents of $p$-Laplace and hypergraph Laplace \cite{IKTU21}.

(c) Note that our nonlinear Markov chain setting is more general.
Since in the nonlinear Dirichlet form setting, the associated generator
is required for a kind of reversibility, while our Markov chain does
not require it. Moreover, our underlying space is more general than
the nonlinear Dirichlet form setting, which requires $L^{2}$ space. 
\end{rem}

\subsection{\label{subsec:The-nonlinear-Perron=00003D002013Frobenius}The nonlinear
Perron-Frobenius theory.}

The classical Perron-Frobenius theorem shows that a nonnegative
matrix has a nonnegative eigenvector associated with its spectral
radius, and if the matrix is irreducible then this nonnegative eigenvector
can be chosen strictly positive. There are many nonlinear generalizations.

For example, in \cite{M05}, the author lets $K$ be a proper cone
in $\mathbb{R}^{N}$, that is, $\alpha K\subset K$ for all $\alpha\in\mathbb{R}_{+}$,
it is closed and convex, $K-K=\mathbb{R}^{N}$, and $K\cap-K=\left\{ 0\right\} $.
Then $K$ induces a partial ordering $x\leq y$ on $K$ defined by
$x-y\in K$. Consider maps satisfying:

(M1) $\Lambda:K\to K,K{^\circ}\to K{^\circ}.$

(M2) $\Lambda(\alpha x)=\alpha\Lambda(x)$ for all $\alpha\geq0$
and $x\in K$ .

(M3) $x\leq y$ implies $\Lambda(x)\leq\Lambda(y)$ for all $x,y\in K$.

(M4) $\Lambda$ is locally Lipschitz continuous near $0$.

Sufficient conditions for the existence and uniqueness of eigenvectors
in the interior of a cone $K$ are developed even when eigenvectors
at the boundary of the cone exist \cite[Theorem 25, Theorem 28]{M05}. 
\begin{thm}
\label{thm: nonlinear Perron=00003D002013Frobenius}Let $K$ be the
positive function set $\mathbb{R}_{>0}^{N}$ and $$Pf\coloneqq\frac{1}{2}\text{log}\left(\text{exp}(f)\cdot\Lambda\left(\text{exp}(f)\right)\right),$$
where $\Lambda$ is defined as above. If $P^{n}f$ has an accumulation
point at infinity (8), then $P^{n}f$ converges. 
\end{thm}
\begin{proof}
Since $\Lambda$ satisfies (M2) and (M3), then $\tilde{P}f\coloneqq\text{log}\left(\Lambda\left(\text{exp}(f)\right)\right)$
satisfies constant additivity (4) and monotonicity (1). And $$Pf=\frac{f+\tilde{P}f}{2}=\frac{1}{2}\text{log}\left(\text{exp}(f)\cdot\Lambda\left(\text{exp}(f)\right)\right)$$
satisfies constant additivity (4) and strict monotonicity of corresponding components (2). Then
we can apply Theorem \ref{thm:NLMC convergence 1} to get the convergence
result with the accumulation points assumption. 
\end{proof}
We next introduce a nonlinear generalization that can be applicable to a specific case relevant to the Ollivier Ricci flow. In
\cite{B09}, the author considers maps $f_{\mathcal{K}}(v)=\text{min}_{A\in\mathcal{K}}Av$,
where $\mathcal{K}$ is a finite set of nonnegative matrices and "min"
means component-wise minimum. In particular, he shows the existence
of nonnegative generalized eigenvectors of $f_{\mathcal{K}}$, and provides necessary and sufficient conditions for the existence of a strictly positive eigenvector. These results apply to our Ollivier Ricci flow (\ref{eq: LLL curvature flow}) and cover the non-connected case. However, the long-term behavior is not addressed.

\section{The Ollivier Ricci curvature of nonlinear Markov chains}

In this section, we introduce a definition of Ollivier Ricci curvature
of nonlinear Markov chains according to the Lipschitz decay property.
Then we can get the convergence results for the nonlinear Markov chain
with a nonnegative Ollivier Ricci curvature. And we can also define
the Laplacian separation flow of a nonlinear Markov chain with $Ric_{1}(P,d)\geq0$.
Then several examples show that the definition is consistent with
the classical Ollivier Ricci curvature (\ref{Ollivier_curvature}),
sectional curvature \cite{CMS24}, coarse Ricci curvature on hypergraphs
\cite{IKTU21} and the modified Ollivier Ricci curvature $\hat{k}_{p}$
for $p$-Laplace (\ref{eq: modified Ollivier Ricci curvature}). 
\begin{defn}
Let $P:\mathbb{R}^{V}\to\mathbb{R}^{V}$ be a nonlinear Markov chain on $G=(V,E)$ with (1) monotonicity, (2) strict monotonicity of corresponding components and (4) constant additivity,
and let $d:V^{2}\to[0,+\infty)$ be the distance function. For $r>0$,
define 
\[
Ric_{r}(P,d)\coloneqq1-\underset{Lip(f)\leq r}{\text{sup}}\frac{Lip(Pf)}{r}.
\]
That is, if $Lip(f)=r$, then $Lip(Pf)\leq(1-Ric_{r})Lip(f)$. 
\end{defn}
By Theorem \ref{thm:NLMC convergence 1}, we can get the following
corollary. 
\begin{cor}
\label{cor:1}Let $r>0$, and assume $(P,d)$ is a nonlinear Markov
chain with $Ric_{r}\geq0$. Let $x_{0}\in V$. Then for all $f\in\mathbb{R}^{V}$ with
$Lip(f)\leq r$, there exists $g\in\mathbb{R}^{V}$ such that 
\[
P^n f - P^n f(x_0) \cdot \vec{1} \to g
\quad \text{and} \quad
P^n f - P^{n-1} f \to \text{const}
\quad \text{as } n \to \infty.
\]
In particular, \( Pg = g + \text{const} \cdot \vec{1} \).
\end{cor}
\begin{proof}
By the definition of $Ric_{r}\geq0$, we can get the accumulation
point at infinity (8) as $Lip(P^{n}f)\leq r$ for all $n$, and by
compactness. Applying Theorem~\ref{thm:NLMC convergence 1}, the result follows.
\end{proof}
Then we want to define the Laplacian separation flow on a nonlinear
Markov chain $(P,d)$ with $Ric_{1}(P,d)\geq0$. Let $V=X\cup K\cup Y$, where $K$ is finite, and suppose $d$ such that
$d(x,y)=\underset{z\in K}{\text{inf}}d(x,z)+d(z,y)$ for
all $x\in X$ and $y\in Y$. Intuitively that means that $K$ separates
$X$ from $Y$. Recall the extremal 1-Lipschitz extension operator
defined as $S:\mathbb{R}^{K}\to\mathbb{R}^{V}$, 
\[
Sf(x)\coloneqq\begin{cases}
\begin{array}{c}
f(x):\\
\underset{y\in K}{\text{min}}\left(f(y)+d(x,y)\right):\\
\underset{y\in K}{\text{max}}\left(f(y)-d(x,y)\right):
\end{array} & \begin{array}{c}
x\in K,\\
x\in Y,\\
x\in X.
\end{array}\end{cases}
\]
Then $S(Lip(1,K))\subseteq Lip(1,V)$. Next, we can get the following
lemma. 
\begin{lem}
\label{lem:Laplacian separation flow of NLMC}Assume $(P,d)$ is a
nonlinear Markov chain with $Ric_{1}(P,d)\geq0$. Define $\tilde{P}:\mathbb{R}^{K}\to\mathbb{R}^{K}$
as $\tilde{P}f=\left(PSf\right)\mid_{K}$. Then $Ric_{1}(\tilde{P},d\mid_{K\times K})\geq0$. 
\end{lem}
\begin{proof}
Since $S(Lip(1,K))\subseteq Lip(1,V)$, i.e., $Sf\in Lip(1,V)$, and
by the definition of $Ric_{1}(P,d)\geq0$, we can get $Ric_{1}(\tilde{P},d\mid_{K\times K})\geq0$. 
\end{proof}
Combining Corollary \ref{cor:1}, we can get the Laplacian separation
result on the nonlinear Markov chain. 
\begin{cor}
\label{cor:2}Let $(P,d)$ be a nonlinear Markov chain with $V=X\cup K\cup Y$.
Assume $Ric_{1}(P,d)\geq0$, then there exist $f\in\mathbb{R}^{V}$
and $C\in\mathbb{R}$ such that $f=Sf\coloneqq S(f\mid_{K})$ and
\[
\Delta f\begin{cases}
\begin{array}{c}
=C,\:\text{on }K,\\
\leq C,\:\text{on }Y,\\
\geq C,\:\text{on }X,
\end{array}\end{cases}
\]
where $\Delta\coloneqq P-id$. 
\end{cor}
\begin{proof}
By Corollary \ref{cor:1}, there exists $g\in Lip(1,K)$ such that
on $K$, 
\[
PSg=g+\text{const}\cdot \overrightarrow{1}.
\]
Let $f=Sg$. Clearly, $f=S(f\mid_{K})$, and on $K$, 
\[
Pf=f+\text{const}\cdot \overrightarrow{1},
\]
i.e., $\Delta f=\text{const}\cdot \overrightarrow{1}$. Moreover, 
\[
SPf=SPSg=Sg+S(\text{const}\cdot \overrightarrow{1})=f+S(\text{const}\cdot \overrightarrow{1}).
\]
Then on $X$, we have $SPf\leq Pf$ as $S$ is the minimum Lipschitz
extension on $X$. Hence, $Pf\geq f+\text{const}\cdot \overrightarrow{1}$, i.e., $\Delta f\geq \text{const}\cdot \overrightarrow{1}$ on
$X$. Similarly, $\Delta f\leq \text{const}\cdot \overrightarrow{1}$ on $Y$, finishing the proof. 
\end{proof}
Next, the following examples show that our Ollivier Ricci curvature
definition is consistent with other settings. 
\begin{example}
(a) Let $P$ be a linear Markov chain, then $Ric_{r}$ is the classical
Ollivier Ricci curvature $\kappa$, see definition (\ref{Ollivier_curvature}).

(b) Let $\tilde{P}$ be a linear Markov chain and define $P(\cdot)=\text{log}\tilde{P}\text{exp}(\cdot)$,
then $Ric_{r}(P,d)\geq0$ for all $r>0$ if the sectional curvature
$\kappa_{sec}\geq0$, see \cite{CMS24}.

(c) Let $P$ be the resolvent of hypergraph Laplace, then $Ric_{r}\geq0$
for all $r>0$ if the coarse Ricci curvature of hypergraphs $\kappa\geq0$,
see \cite{IKTU21}.

(d) Let $P$ be the resolvent of $p$-Laplace, then $Ric_{r}\geq0$
for all $r>0$ if the modified Ollivier Ricci curvature of $p$-Laplace
$\hat{k}_{p}\geq0$, see definition (\ref{eq: modified Ollivier Ricci curvature})
in the introduction. 
\end{example}
\begin{rem}
For the above examples (a)-(d) with $Ric_{1}(P,d)\geq0$, by Corollary
\ref{cor:2} the Laplacian separation flow can be defined respectively. 
\end{rem}
$\mathbf{\boldsymbol{\mathbf{Acknowledgements}\mathbf{}\text{:}}}$
The authors would like to thank J\"urgen Jost, Bobo Hua and Tao Wang
for helpful discussions and suggestions. The authors sincerely thank for the referee's valuable feedback and for helping us enhance our work.

\bibliographystyle{plain}
\bibliography{NLMC_ref_copy}

@book {J95,
    AUTHOR = {Jost, J\"{u}rgen},
     TITLE = {Riemannian geometry and geometric analysis},
    SERIES = {Universitext},
 PUBLISHER = {Springer-Verlag, Berlin},
      YEAR = {1995},
     PAGES = {xii+401},
      ISBN = {3-540-57113-2},
   MRCLASS = {53C21 (53-01 53C20 58-01 58E05 58E10)},
  MRNUMBER = {1351009},
MRREVIEWER = {Man\ Chun\ Leung},
       DOI = {10.1007/978-3-662-03118-6},
       URL = {https://doi.org/10.1007/978-3-662-03118-6},
}

@article {LY10,
    AUTHOR = {Lin, Yong and Yau, Shing-Tung},
     TITLE = {Ricci curvature and eigenvalue estimate on locally finite
              graphs},
   JOURNAL = {Math. Res. Lett.},
  FJOURNAL = {Mathematical Research Letters},
    VOLUME = {17},
      YEAR = {2010},
    NUMBER = {2},
     PAGES = {343--356},
      ISSN = {1073-2780},
   MRCLASS = {05C10 (05C50)},
  MRNUMBER = {2644381},
       DOI = {10.4310/MRL.2010.v17.n2.a13},
       URL = {https://doi.org/10.4310/MRL.2010.v17.n2.a13},
}

@incollection {S99,
    AUTHOR = {Schmuckenschl\"{a}ger, Michael},
     TITLE = {Curvature of nonlocal {M}arkov generators},
 BOOKTITLE = {Convex geometric analysis ({B}erkeley, {CA}, 1996)},
    SERIES = {Math. Sci. Res. Inst. Publ.},
    VOLUME = {34},
     PAGES = {189--197},
 PUBLISHER = {Cambridge Univ. Press, Cambridge},
      YEAR = {1999},
      ISBN = {0-521-64259-0},
   MRCLASS = {60J75 (52C99 60J35)},
  MRNUMBER = {1665591},
MRREVIEWER = {Joerg-Uwe\ Loebus},
}

@article {JL14,
    AUTHOR = {Jost, J\"{u}rgen and Liu, Shiping},
     TITLE = {Ollivier's {R}icci curvature, local clustering and
              curvature-dimension inequalities on graphs},
   JOURNAL = {Discrete Comput. Geom.},
  FJOURNAL = {Discrete \& Computational Geometry. An International Journal
              of Mathematics and Computer Science},
    VOLUME = {51},
      YEAR = {2014},
    NUMBER = {2},
     PAGES = {300--322},
      ISSN = {0179-5376,1432-0444},
   MRCLASS = {05C62 (53A10 60J05)},
  MRNUMBER = {3164168},
       DOI = {10.1007/s00454-013-9558-1},
       URL = {https://doi.org/10.1007/s00454-013-9558-1},
}

@article {O09,
    AUTHOR = {Ollivier, Yann},
     TITLE = {Ricci curvature of {M}arkov chains on metric spaces},
   JOURNAL = {J. Funct. Anal.},
  FJOURNAL = {Journal of Functional Analysis},
    VOLUME = {256},
      YEAR = {2009},
    NUMBER = {3},
     PAGES = {810--864},
      ISSN = {0022-1236,1096-0783},
   MRCLASS = {58J65 (46E35 53C23 60B05 60J10)},
  MRNUMBER = {2484937},
MRREVIEWER = {Mu\ Fa\ Chen},
       DOI = {10.1016/j.jfa.2008.11.001},
       URL = {https://doi.org/10.1016/j.jfa.2008.11.001},
}

@article {O07,
    AUTHOR = {Ollivier, Yann},
     TITLE = {Ricci curvature of metric spaces},
   JOURNAL = {C. R. Math. Acad. Sci. Paris},
  FJOURNAL = {Comptes Rendus Math\'{e}matique. Acad\'{e}mie des Sciences.
              Paris},
    VOLUME = {345},
      YEAR = {2007},
    NUMBER = {11},
     PAGES = {643--646},
      ISSN = {1631-073X,1778-3569},
   MRCLASS = {53C23},
  MRNUMBER = {2371483},
MRREVIEWER = {Goulnara\ N.\ Arzhantseva},
       DOI = {10.1016/j.crma.2007.10.041},
       URL = {https://doi.org/10.1016/j.crma.2007.10.041},
}

@article {EM12,
    AUTHOR = {Erbar, Matthias and Maas, Jan},
     TITLE = {Ricci curvature of finite {M}arkov chains via convexity of the
              entropy},
   JOURNAL = {Arch. Ration. Mech. Anal.},
  FJOURNAL = {Archive for Rational Mechanics and Analysis},
    VOLUME = {206},
      YEAR = {2012},
    NUMBER = {3},
     PAGES = {997--1038},
      ISSN = {0003-9527,1432-0673},
   MRCLASS = {58J65 (49Q20 60B05 60J05)},
  MRNUMBER = {2989449},
MRREVIEWER = {Pedro\ J.\ Catuogno},
       DOI = {10.1007/s00205-012-0554-z},
       URL = {https://doi.org/10.1007/s00205-012-0554-z},
}

@article {M13,
    AUTHOR = {Mielke, Alexander},
     TITLE = {Geodesic convexity of the relative entropy in reversible
              {M}arkov chains},
   JOURNAL = {Calc. Var. Partial Differential Equations},
  FJOURNAL = {Calculus of Variations and Partial Differential Equations},
    VOLUME = {48},
      YEAR = {2013},
    NUMBER = {1-2},
     PAGES = {1--31},
      ISSN = {0944-2669,1432-0835},
   MRCLASS = {60J27 (53C21 53C23 82B35)},
  MRNUMBER = {3090532},
MRREVIEWER = {Roman\ Urban},
       DOI = {10.1007/s00526-012-0538-8},
       URL = {https://doi.org/10.1007/s00526-012-0538-8},
}

@article {LLY11,
    AUTHOR = {Lin, Yong and Lu, Linyuan and Yau, Shing-Tung},
     TITLE = {Ricci curvature of graphs},
   JOURNAL = {Tohoku Math. J. (2)},
  FJOURNAL = {The Tohoku Mathematical Journal. Second Series},
    VOLUME = {63},
      YEAR = {2011},
    NUMBER = {4},
     PAGES = {605--627},
      ISSN = {0040-8735,2186-585X},
   MRCLASS = {05C99 (05C76 05C80 05C81)},
  MRNUMBER = {2872958},
       DOI = {10.2748/tmj/1325886283},
       URL = {https://doi.org/10.2748/tmj/1325886283},
}

@article{NLLG19,
  title={Community detection on networks with {R}icci flow},
  author={Ni, Chien-Chun and Lin, Yu-Yao and Luo, Feng and Gao, Jie},
  journal={Scientific reports},
  volume={9},
  number={1},
  pages={9984},
  year={2019},
  publisher={Nature Publishing Group UK London}
}

@article{BLLWY20,
      title={Ollivier {R}icci-flow on weighted graphs}, 
      author={Bai, Shuliang and Lin, Yong and Lu,  Linyuan and Wang, Zhiyu and Yau, Shing-Tung},
      year={2024},
       journal={American Journal of Mathematics},
    VOLUME = {146(6)},
     PAGES = {1723-1747},
       DOI = {10.1353/ajm.2024.a944362},
       URL = {https://dx.doi.org/10.1353/ajm.2024.a944362},

     
}

@article {MW19,
    AUTHOR = {M\"{u}nch, Florentin and Wojciechowski, Rados\l aw K.},
     TITLE = {Ollivier {R}icci curvature for general graph {L}aplacians:
              heat equation, {L}aplacian comparison, non-explosion and
              diameter bounds},
   JOURNAL = {Adv. Math.},
  FJOURNAL = {Advances in Mathematics},
    VOLUME = {356},
      YEAR = {2019},
     PAGES = {106759, 45},
      ISSN = {0001-8708,1090-2082},
   MRCLASS = {53C21 (05C12 05C63 58J65 60J27)},
  MRNUMBER = {3998765},
MRREVIEWER = {Xueping\ Huang},
       DOI = {10.1016/j.aim.2019.106759},
       URL = {https://doi.org/10.1016/j.aim.2019.106759},
}

@article{HM21,
url = {https://doi.org/10.1515/crelle-2025-0003},
title = {Every salami has two ends},
author = {Hua, Bobo and M\"{u}nch, Florentin},
pages = {291--321},
volume = {2025},
number = {821},
journal = {Journal f{\"u}r die reine und angewandte Mathematik (Crelles Journal)},
doi = {doi:10.1515/crelle-2025-0003},
year = {2025},

}

@article{HMZ23,
      title={Some variants of discrete positive mass theorems on graphs}, 
      author={Hua, Bobo and M\"{u}nch, Florentin and Zhang, Haohang},
      year={2023},
       journal={arXiv: 2307.08334},
     
}

@book{K10,
  title={{N}onlinear {M}arkov processes and kinetic equations},
  author={Kolokoltsov, Vassili N},
  volume={182},
  year={2010},
  publisher={Cambridge University Press}
}

@article {M66,
    AUTHOR = {McKean, Jr., H. P.},
     TITLE = {A class of {M}arkov processes associated with nonlinear
              parabolic equations},
   JOURNAL = {Proc. Nat. Acad. Sci. U.S.A.},
  FJOURNAL = {Proceedings of the National Academy of Sciences of the United
              States of America},
    VOLUME = {56},
      YEAR = {1966},
     PAGES = {1907--1911},
      ISSN = {0027-8424},
   MRCLASS = {60.62},
  MRNUMBER = {221595},
MRREVIEWER = {F.\ B.\ Knight},
       DOI = {10.1073/pnas.56.6.1907},
       URL = {https://doi.org/10.1073/pnas.56.6.1907},
}

@article{M23,
      title={Ollivier curvature, {I}soperimetry, concentration, and {L}og-{S}obolev inequalitiy}, 
      author={M\"{u}nch, Florentin},
      year={2023},
       journal={arXiv: 2309.06493},
     
}

@book{S09,
  title={Basics of applied stochastic processes},
  author={Serfozo, Richard},
  year={2009},
  publisher={Springer Science \& Business Media}
}

@article {N23,
    AUTHOR = {Neumann, Berenice Anne},
     TITLE = {{N}onlinear {M}arkov chains with finite state space: invariant
              distributions and long-term behaviour},
   JOURNAL = {J. Appl. Probab.},
  FJOURNAL = {Journal of Applied Probability},
    VOLUME = {60},
      YEAR = {2023},
    NUMBER = {1},
     PAGES = {30--44},
      ISSN = {0021-9002,1475-6072},
   MRCLASS = {60G65},
  MRNUMBER = {4546109},
       DOI = {10.1017/jpr.2022.23},
       URL = {https://doi.org/10.1017/jpr.2022.23},
}

@book {KM19,
    AUTHOR = {Kolokoltsov, Vassili N. and Malafeyev, Oleg A.},
     TITLE = {Many agent games in socio-economic systems: corruption,
              inspection, coalition building, network growth, security},
    SERIES = {Springer Series in Operations Research and Financial
              Engineering},
 PUBLISHER = {Springer, Cham},
      YEAR = {2019},
     PAGES = {xix+196},
      ISBN = {978-3-030-12370-3; 978-3-030-12371-0},
   MRCLASS = {91-02 (49N70 60J27 60J28 60K35 91A13 91A15)},
  MRNUMBER = {3929735},
       DOI = {10.1007/978-3-030-12371-0},
       URL = {https://doi.org/10.1007/978-3-030-12371-0},
}

@article {B14,
    AUTHOR = {Butkovsky, O. A.},
     TITLE = {On ergodic properties of nonlinear {M}arkov chains and
              stochastic {M}c{K}ean-{V}lasov equations},
   JOURNAL = {Theory Probab. Appl.},
  FJOURNAL = {Theory of Probability and its Applications},
    VOLUME = {58},
      YEAR = {2014},
    NUMBER = {4},
     PAGES = {661--674},
      ISSN = {0040-585X,1095-7219},
   MRCLASS = {60J10 (37A30 60B10 60H10)},
  MRNUMBER = {3403022},
       DOI = {10.1137/S0040585X97986825},
       URL = {https://doi.org/10.1137/S0040585X97986825},
}

@article {S16,
    AUTHOR = {Saburov, Mansoor},
     TITLE = {{E}rgodicity of nonlinear {M}arkov operators on the finite
              dimensional space},
   JOURNAL = {Nonlinear Anal.},
  FJOURNAL = {Nonlinear Analysis. Theory, Methods \& Applications. An
              International Multidisciplinary Journal},
    VOLUME = {143},
      YEAR = {2016},
     PAGES = {105--119},
      ISSN = {0362-546X,1873-5215},
   MRCLASS = {47H25 (37A30 47H60)},
  MRNUMBER = {3516825},
       DOI = {10.1016/j.na.2016.05.006},
       URL = {https://doi.org/10.1016/j.na.2016.05.006},
}

@article {BCMP18,
    AUTHOR = {Bourne, D. P. and Cushing, D. and Liu, S. and M\"{u}nch, F.
              and Peyerimhoff, N.},
     TITLE = {{O}llivier-{R}icci idleness functions of graphs},
   JOURNAL = {SIAM J. Discrete Math.},
  FJOURNAL = {SIAM Journal on Discrete Mathematics},
    VOLUME = {32},
      YEAR = {2018},
    NUMBER = {2},
     PAGES = {1408--1424},
      ISSN = {0895-4801,1095-7146},
   MRCLASS = {05C62 (51K10 90C08)},
  MRNUMBER = {3815539},
       DOI = {10.1137/17M1134469},
       URL = {https://doi.org/10.1137/17M1134469},
}

@book{RW09,
  title={Variational analysis},
  author={Rockafellar, R Tyrrell and Wets, Roger J-B},
  volume={317},
  year={2009},
  publisher={Springer Science \& Business Media}
}

@book{M92,
  title={Nonlinear semigroups},
  author={Miyadera, Isao},
  volume={109},
  year={1992},
  publisher={American Mathematical Soc.}
}

@incollection {J98,
    AUTHOR = {Jost, J\"{u}rgen},
     TITLE = {{N}onlinear {D}irichlet forms},
 BOOKTITLE = {New directions in {D}irichlet forms},
    SERIES = {AMS/IP Stud. Adv. Math.},
    VOLUME = {8},
     PAGES = {1--47},
 PUBLISHER = {Amer. Math. Soc., Providence, RI},
      YEAR = {1998},
      ISBN = {0-8218-1061-8},
   MRCLASS = {31C25},
  MRNUMBER = {1652278},
MRREVIEWER = {Zhongmin\ Qian},
       DOI = {10.1090/amsip/008/01},
       URL = {https://doi.org/10.1090/amsip/008/01},
}

@book{CB21,
  title={{N}onlinear {D}irichlet {F}orms},
  author={Claus, Burkhard},
  year={2021},
   publisher={Dissertation, Dresden, Technische Universit{\"a}t Dresden}
}

@article{IKTU21,
      title={{C}oarse {R}icci curvature of hypergraphs and its generalization}, 
      author={Ikeda, MasaHiro and Kitabeppu, Yu and Takai, Yuuki and Uehara,Takato},
      year={2021},
       journal={arXiv: 2102.00698},
     
}

@article {M05,
    AUTHOR = {Metz, Volker},
     TITLE = {Nonlinear {P}erron-{F}robenius theory in finite dimensions},
   JOURNAL = {Nonlinear Anal.},
  FJOURNAL = {Nonlinear Analysis. Theory, Methods \& Applications. An
              International Multidisciplinary Journal},
    VOLUME = {62},
      YEAR = {2005},
    NUMBER = {2},
     PAGES = {225--244},
      ISSN = {0362-546X,1873-5215},
   MRCLASS = {47J10 (15A48 37C99 47H99)},
  MRNUMBER = {2145604},
MRREVIEWER = {Peter\ E.\ Kloeden},
       DOI = {10.1016/j.na.2005.02.116},
       URL = {https://doi.org/10.1016/j.na.2005.02.116},
}

@article {B09,
    AUTHOR = {Bondarenko, Ievgen},
     TITLE = {Dynamics of piecewise linear maps and sets of nonnegative
              matrices},
   JOURNAL = {Linear Algebra Appl.},
  FJOURNAL = {Linear Algebra and its Applications},
    VOLUME = {431},
      YEAR = {2009},
    NUMBER = {5-7},
     PAGES = {495--510},
      ISSN = {0024-3795,1873-1856},
   MRCLASS = {15B48},
  MRNUMBER = {2535527},
MRREVIEWER = {Oliver\ Patrick\ Mason},
       DOI = {10.1016/j.laa.2009.02.038},
       URL = {https://doi.org/10.1016/j.laa.2009.02.038},
}

@book {LP17,
    AUTHOR = {Levin, David A. and Peres, Yuval},
     TITLE = {Markov chains and mixing times},
   EDITION = {Second},
      NOTE = {With contributions by Elizabeth L. Wilmer,
              With a chapter on ``Coupling from the past'' by James G. Propp
              and David B. Wilson},
 PUBLISHER = {American Mathematical Society, Providence, RI},
      YEAR = {2017},
     PAGES = {xvi+447},
      ISBN = {978-1-4704-2962-1},
   MRCLASS = {60J10 (60-01 60B15 60C05 60J27 60K35 68U20 82C22)},
  MRNUMBER = {3726904},
       DOI = {10.1090/mbk/107},
       URL = {https://doi.org/10.1090/mbk/107},
}

@article{CMS24,
      title={Entropy and curvature: beyond the {P}eres-{T}etali conjecture}, 
      author={Caputo, Pietro and  M\"{u}nch, Florentin and Salez, Justin },
      year={2024},
       journal={arXiv: 2401.17148},
     
}

@article {M18,
    AUTHOR = {M\"{u}nch, Florentin},
     TITLE = {{L}i-{Y}au inequality on finite graphs via non-linear curvature
              dimension conditions},
   JOURNAL = {J. Math. Pures Appl. (9)},
  FJOURNAL = {Journal de Math\'ematiques Pures et Appliqu\'ees. Neuvi\`eme
              S\'erie},
    VOLUME = {120},
      YEAR = {2018},
     PAGES = {130--164},
      ISSN = {0021-7824,1776-3371},
   MRCLASS = {35R02 (05C10 35K05 53C21 58J35)},
  MRNUMBER = {3906157},
MRREVIEWER = {Thierry\ Coulhon},
       DOI = {10.1016/j.matpur.2018.10.006},
       URL = {https://doi.org/10.1016/j.matpur.2018.10.006},
}

@article {LL18,
    AUTHOR = {Lin, Yong and Liu, Shuang},
     TITLE = {Equivalent properties of {CD} inequalities on graphs},
   JOURNAL = {Acta Math. Sinica (Chinese Ser.)},
  FJOURNAL = {Acta Mathematica Sinica. Chinese Series},
    VOLUME = {61},
      YEAR = {2018},
    NUMBER = {3},
     PAGES = {431--440},
      ISSN = {0583-1431},
   MRCLASS = {58J35 (53C23)},
  MRNUMBER = {4545901},
}

@article{m19li,
	title={{L}i-{Y}au inequality under $ CD (0, n) $ on graphs},
	author={M\"{u}nch, Florentin},
	journal={arXiv preprint arXiv:1909.10242},
	year={2019}
}

@article{dier2021discrete,
	title={Discrete versions of the {L}i-{Y}au gradient estimate},
	author={Dier, Dominik and Kassmann, Moritz and Zacher, Rico},
	journal={Annali Scuola Normale Superiore-Classe di Scienze},
	pages={691--744},
	year={2021}
}

@article{weber2023li,
	title={{L}i-{Y}au inequalities for general non-local diffusion equations via reduction to the heat kernel},
	author={Weber, Frederic and Zacher, Rico},
	journal={Mathematische Annalen},
	volume={385},
	number={1},
	pages={393--419},
	year={2023},
	publisher={Springer}
}

@article{horn2019volume,
	title={Volume doubling, {P}oincar{\'e} inequality and {G}aussian heat kernel estimate for non-negatively curved graphs},
	author={Horn, Paul and Lin, Yong and Liu, Shuang and Yau, Shing-Tung},
	journal={Journal f{\"u}r die reine und angewandte Mathematik (Crelles Journal)},
	volume={2019},
	number={757},
	pages={89--130},
	year={2019},
	publisher={De Gruyter}
}

@article{forman2003bochner,
	title={Bochner's method for cell complexes and combinatorial Ricci curvature},
	author={Forman},
	journal={Discrete \& Computational Geometry},
	volume={29},
	pages={323--374},
	year={2003},
	publisher={Springer}
}

@article {Hamilton1982,
    AUTHOR = {Hamilton, Richard S.},
     TITLE = {Three-manifolds with positive {R}icci curvature},
   JOURNAL = {J. Differential Geometry},
  FJOURNAL = {Journal of Differential Geometry},
    VOLUME = {17},
      YEAR = {1982},
    NUMBER = {2},
     PAGES = {255--306},
      ISSN = {0022-040X,1945-743X},
   MRCLASS = {53C25 (35K55 58G30)},
  MRNUMBER = {664497},
MRREVIEWER = {J.\ L.\ Kazdan},
       URL = {http://projecteuclid.org/euclid.jdg/1214436922},
}

\end{document}